\documentclass[11pt,reqno]{amsart}

\usepackage[T1]{fontenc}
\usepackage[utf8]{inputenc}
\usepackage{amsmath,amssymb,amsthm,mathtools}
\usepackage{graphicx}
\usepackage{booktabs}
\usepackage{enumitem}
\usepackage[expansion=false,protrusion=true]{microtype}
\usepackage{tikz}
\usetikzlibrary{arrows.meta,decorations.markings,patterns}
\usepackage[margin=1.15in]{geometry}
\usepackage[colorlinks=true,linkcolor=blue!55!black,citecolor=blue!55!black,
            urlcolor=blue!55!black]{hyperref}

\theoremstyle{plain}
\newtheorem{theorem}{Theorem}[section]
\newtheorem{proposition}[theorem]{Proposition}
\newtheorem{lemma}[theorem]{Lemma}
\newtheorem{corollary}[theorem]{Corollary}
\theoremstyle{definition}
\newtheorem{definition}[theorem]{Definition}
\newtheorem{remark}[theorem]{Remark}

\newcommand{\C}{\mathbb{C}}
\newcommand{\R}{\mathbb{R}}
\newcommand{\Z}{\mathbb{Z}}

\newcommand{\Q}{\mathbb{Q}}
\newcommand{\Rea}{\operatorname{Re}}
\newcommand{\Ima}{\operatorname{Im}}
\newcommand{\Harc}{\textup{(H$_{\mathrm{arc}}$)}}
\newcommand{\Res}{\operatorname*{Res}}
\newcommand{\Gs}{\mathcal{G}}
\newcommand{\Lp}{\Lambda^{+}}
\newcommand{\Lm}{\Lambda^{-}}
\newcommand{\Kr}{K^{\rightarrow}}
\newcommand{\Kl}{K^{\leftarrow}}
\newcommand{\Ku}{K^{\uparrow}}
\newcommand{\Kd}{K^{\downarrow}}
\newcommand{\va}{\mathbf{a}}
\newcommand{\vb}{\mathbf{b}}
\newcommand{\eps}{\varepsilon}

\begin{document}

\title[The Meijer--Barnes $K$-function for $N<0$]
      {Completing the Meijer--Barnes $K$-function:\\ the excluded case $N<0$}

\author{Changhao He}
\address{CEMSE Division, King Abdullah University of Science and Technology (KAUST),
Thuwal 23955-6900, Kingdom of Saudi Arabia}
\email{changhao.he@kaust.edu.sa}

\date{\today\\[2pt]{\small\textsc{version 7}\ \ ---\ \ all references to
\cite{KK2024} are to the published version, \emph{Constr.\ Approx.} \textbf{63}
(2026), 223--256}}


\begin{abstract}
Karp and Kuznetsov \cite{KK2024} introduced the \emph{Meijer--Barnes
$K$-function} by replacing every Euler gamma factor in the Mellin--Barnes kernel
of the Meijer $G$-function by a Barnes double gamma factor $G(z;\tau)$. Their
definition is organised by the integer $N=2(m+n)-p-q$ and \emph{explicitly
excludes} the case $N<0$, on the grounds that four mutually non-homotopic
contours are then admissible, so that the defining integral does not single out
one function.

This paper supplies the missing case. We show that for $N<0$ the integrand
decays super-exponentially in the two \emph{horizontal} sectors, so that all additional convergence restrictions from the original definition
disappear: the integral
converges absolutely for every $z$ on the Riemann surface of the logarithm and
every $\alpha\in\C$. Under a precise notion of admissible contour there are then
exactly four classes, $\Kr,\Kl,\Ku,\Kd$; they satisfy the universal linear
relation $\Ku+\Kd=\Kr+\Kl$ and therefore span a space of dimension at most
three, which for real data reduces to at most three real degrees of freedom with
$\Kd=\overline{\Ku}$. We obtain residue expansions for $\Kr$ and $\Kl$, and
vanishing theorems strictly stronger than the $N\ge0$ analogue, under an
explicit closing-arc growth hypothesis, motivated by Diophantine small-denominator
phenomena, that is needed to control the
arcs on which the known kernel asymptotic does not apply; and we show that all three transformation laws persist
unconditionally, with the inversion law interchanging $\Kr$ and $\Kl$ while
fixing $\Ku,\Kd$.

Two further structural results are proved. First, no vertical contour is
admissible when $N<0$, so the standard vertical-contour Mellin-transform
machinery of \cite{KK2024} cannot be transferred directly; its role is taken by a compatible two-index recurrence on the
residue coefficients, driven by the very same gamma ratios. Second, the classical
stability properties of the Meijer $G$-function reduce to four elementary
insertion rules, which leave $N$ strictly invariant while Mellin convolution
merely adds it; since Meijer $G$ embeds with $N=0$, no composition of the
classical transform calculus applied to hypergeometric input can leave the
stratum $N=0$. This explains why every application known in the literature has
$N=0$, and indicates where an $N\ne0$ example would have to come from. We
distinguish carefully between the kernel-level calculus, which is unconditional,
and the analytic transform identities, which we establish for the operators and
contour classes where the required interchange is legitimate.

Representative identities and parameter instances are validated numerically;
most algebraic and contour identities reach $10^{-17}$ or better, while the
quadrature/differentiation checks are limited to $10^{-8}$--$10^{-13}$. For this we construct an evaluator for $G(z;\tau)$ at
general $\tau>0$ --- to the best of our knowledge no widely used numerical
library provides one; \texttt{mpmath.barnesg} is the case $\tau=1$
--- which requires expansion coefficients not printed in the literature. Along
the way we record a sign correction to \cite{KK2024}.
\end{abstract}

\maketitle

\section{Introduction}

\subsection{Background}
The Meijer $G$-function
\begin{equation}\label{eq:meijerG}
G^{m,n}_{p,q}\!\left(\begin{matrix}a_1,\dots,a_p\\ b_1,\dots,b_q\end{matrix}\,
\middle|\,z\right)
=\frac{1}{2\pi i}\int_{L}
\frac{\prod_{j=1}^{m}\Gamma(b_j-s)\prod_{j=1}^{n}\Gamma(1-a_j+s)}
     {\prod_{j=m+1}^{q}\Gamma(1-b_j+s)\prod_{j=n+1}^{p}\Gamma(a_j-s)}\,z^{s}\,ds
\end{equation}
owes its universality to a single structural fact: its Mellin transform is a
product and quotient of gamma functions, and multiplying or dividing such an
object by further gamma factors returns an object of the same type. Equivalently,
a Mellin transform $M$ arising from a $G$-function obeys a first-order difference
equation $M(s+1)=R(s)M(s)$ with $R$ \emph{rational}.

Over the last fifteen years a family of probabilistic objects has emerged whose
Mellin transforms obey the same kind of difference equation with $R$ a
\emph{ratio of gamma functions} rather than a rational function
\cite{Kuznetsov2011,KuznetsovKwasnicki2018,KuznetsovPardo2013,Ostrovsky2013,%
JedidiSimonWang2018}. Solving such a recursion requires a function one rung
above $\Gamma$, namely the Barnes double gamma function $G(z;\tau)$
\cite{Barnes1899,Barnes1901}, which satisfies
\begin{equation}\label{eq:quasiper}
G(z+1;\tau)=\Gamma\!\left(\tfrac{z}{\tau}\right)G(z;\tau),
\qquad
G(z+\tau;\tau)=(2\pi)^{\frac{\tau-1}{2}}\tau^{\frac12-z}\,\Gamma(z)\,G(z;\tau),
\end{equation}
i.e.\ whose shift multiplier is itself a gamma function.

Karp and Kuznetsov \cite{KK2024} unified these examples by replacing every
$\Gamma$ in \eqref{eq:meijerG} with a $G(\cdot;\tau)$. Writing
$\va=(a_1,\dots,a_p)\in\C^p$, $\vb=(b_1,\dots,b_q)\in\C^q$,
$0\le m\le q$, $0\le n\le p$, and
\begin{equation}\label{eq:Gscript}
\Gs(s)=\Gs^{m,n}_{p,q}\!\left(\begin{matrix}\va\\ \vb\end{matrix}\,\middle|\,s;\tau\right)
:=\frac{\prod_{j=1}^{m}G(b_j-s;\tau)\prod_{j=1}^{n}G(1+\tau-a_j+s;\tau)}
       {\prod_{j=m+1}^{q}G(1+\tau-b_j+s;\tau)\prod_{j=n+1}^{p}G(a_j-s;\tau)},
\end{equation}
\begin{equation}\label{eq:phi}
\phi(s):=\Gs(s)\,e^{\pi\alpha s^{2}/\tau},\qquad \alpha\in\C,
\end{equation}
they define the \emph{Meijer--Barnes $K$-function} by
\begin{equation}\label{eq:Kdef}
K^{m,n}_{p,q}\!\left(\begin{matrix}\va\\ \vb\end{matrix}\,\middle|\,z;\tau,\alpha\right)
:=\frac{1}{2\pi i}\int_{\gamma}\phi(s)\,z^{-s}\,ds .
\end{equation}
Because $\Gamma(s)=(2\pi)^{(1-\tau)/2}\tau^{s-1/2}G(s+\tau;\tau)/G(s;\tau)$,
every Meijer $G$-function --- hence every ${}_pF_q$ --- is a special case of
\eqref{eq:Kdef}.

The behaviour of \eqref{eq:Kdef} is controlled by the integer
\begin{equation}\label{eq:Nmunuxi}
N:=2(m+n)-p-q,
\end{equation}
together with the auxiliary quantities $\mu,\nu,\xi$ recalled in
\S\ref{sec:prelim}. The parameter $N$ is exactly twice the classical Meijer
convergence index $\delta=m+n-\tfrac{p+q}{2}$: for \eqref{eq:meijerG} the kernel
decays on a vertical line like $e^{-\pi\delta|t|}$, whereas for \eqref{eq:Gscript}
it behaves like $\exp\!\big(\tfrac{N}{2\tau}r^{2}\ln r\cos 2\theta\big)$,
$s=re^{i\theta}$.

\subsection{The gap}
Karp and Kuznetsov list seven parameter regimes (log-quadratic, quadratic,
upper-/lower-balanced, balanced, linear, logarithmic), give a table of admissible
contour angles, and then \emph{exclude} several parameter sets, the first being
$N<0$. Their reason is worth quoting:
\begin{quote}\small
``We can choose four contours of integration $L_{-\eps,\eps}$,
$L_{-\eps,\pi-\eps}$, $L_{-\pi+\eps,\eps}$ and $L_{-\pi+\eps,\pi-\eps}$. These
contours can not be transformed into each other\ldots{} Thus in the case $N<0$
we can have up to four different functions that could be defined via (26). While
this case may be of interest, we decided not to include it in the present
paper.'' \cite[p.~235]{KK2024}
\end{quote}
So the obstruction is not divergence. It is that the integral converges
\emph{too well}: no single contour is forced, and \eqref{eq:Kdef} is ambiguous.

\subsection{Contributions}
We resolve the ambiguity by defining all four contour functions and deriving a
universal relation among them, together with the structural relations we can
currently reach. Whether a second, independent relation exists is left open
(item~\ref{op:dim} of \S\ref{sec:open}).

\begin{enumerate}[leftmargin=2.1em,itemsep=3pt]
\item[(C1)] \textbf{Geometry (\S\ref{sec:geometry}).} For $N<0$ the decay sectors
are the two \emph{horizontal} ones, $|\arg s|<\pi/4$ and $|\arg s-\pi|<\pi/4$;
and, because $l\tau+k$ is real, each of the two pole sets $\Lp,\Lm$ lies on
\emph{finitely many horizontal lines}, one comb running east and one west. These
two facts together force exactly four homotopy classes and make the bookkeeping
finite.

\item[(C2)] \textbf{Unconditional convergence (Theorem~\ref{thm:A}).} For $N<0$
every additional convergence restriction imposed in the $N\ge0$ definition of
\cite{KK2024} disappears: no condition on
$\Rea\alpha$, on $\Rea\nu$, or on $\arg\mu$ is needed, and the integral
converges absolutely for all $z$ on the Riemann surface of the logarithm and all
$\alpha\in\C$, defining an entire function of $w=\tfrac{1}{2\pi}\log z$.

\item[(C3)] \textbf{Residue expansions (Theorem~\ref{thm:B}).}
$\Kr=-\sum_{\Lp}\Res$ and $\Kl=+\sum_{\Lm}\Res$, as identities rather than
asymptotic expansions --- under the closing-arc hypothesis
\Harc of \S\ref{sec:arcs}, which is what makes the closing arcs
negligible in the one region the kernel asymptotic of \cite{KK2024} does not
cover. Two forms of that hypothesis are given, a supremum form and a weaker
$L^{1}$ form; the latter needs no prescribed sequence of radii, because
Lemma~\ref{lem:adaptive} selects one from each annulus. The residues are computed in closed form from a formula for $G'$ at every
lattice zero (Lemma~\ref{lem:Gprime}).

\item[(C4)] \textbf{One universal relation (Theorem~\ref{thm:C}).}
$\Ku+\Kd=\Kr+\Kl$, proved unconditionally by realising all four classes on rays
that stay inside the decay sectors. Hence the span has dimension at most three;
for real data $\Kd=\overline{\Ku}$ with $\Rea\Ku$ forced, leaving at most three
real degrees of freedom (Corollary~\ref{cor:real}). Whether the dimension is
exactly three we leave open; see Remark~\ref{rem:three}.

\item[(C5)] \textbf{Stronger vanishing (Theorem~\ref{thm:D}).}
$\Lp=\varnothing\Rightarrow\Kr\equiv0$ and $\Lm=\varnothing\Rightarrow\Kl\equiv0$
\emph{identically}, where the $N=0$ analogue \cite[Thm.~3]{KK2024} only gives
vanishing on a half-line.

\item[(C6)] \textbf{Transformations persist (Theorem~\ref{thm:E}).} The shift and
modular laws hold for each of the four separately; the inversion law interchanges
$\Kr\leftrightarrow\Kl$ and fixes $\Ku,\Kd$. Since $N$ is invariant under
$\tau\mapsto1/\tau$, the class $N<0$ is closed under the modular law.

\item[(C7)] \textbf{No admissible vertical contour, but a recurrence
(Theorems~\ref{thm:F1}, \ref{thm:F}).} The vertical line lies in a growth sector,
so no vertical Mellin--Barnes representation exists and the proofs of
\cite[Thms.~6--7]{KK2024} do not transfer. Their replacement is a compatible
two-index recurrence on the residue coefficients driven by the same gamma ratios
$F,H$ that appear in the proof of \cite[Thm.~7]{KK2024}. We are careful not to
claim that $K^{\bullet}$ has no Mellin transform; see
Remark~\ref{rem:notclaimed}.

\item[(C8)] \textbf{The stability properties, and a strict invariant
(\S\ref{sec:closure}).} Four elementary insertion rules
(Lemma~\ref{lem:insert}) reduce Erd\'elyi--Kober and Riemann--Liouville
fractional integrals and derivatives, the Laplace and Euler transforms,
differentiation of every order and Mellin
convolution to a single mechanism at the level of the \emph{kernel}
(Theorem~\ref{thm:kernelcalc}, unconditional). The corresponding \emph{analytic}
identities are then established operator by operator and class by class
(Theorem~\ref{thm:closure}); two of them we can only record formally.
Differentiation is given to all orders in closed form \eqref{eq:derivk}, and
Riemann--Liouville fractional \emph{derivatives} are obtained by
composing that with the Erd\'elyi--Kober rule, with no continuation in the order
(Theorem~\ref{thm:RL}). The calculus leaves $N$ invariant and Mellin convolution
adds it, whence the ``$N=0$ trap'' of Corollary~\ref{cor:Ntrap}.

\item[(C9)] \textbf{Numerics and a corrigendum (\S\ref{sec:numerics},
\S\ref{sec:errata}).} We build an evaluator for $G(z;\tau)$ at general $\tau>0$,
for which we must derive Binet coefficients not printed in the literature
(Lemma~\ref{lem:binet}), and validate representative identities and parameter
instances, typically to $10^{-17}$ or better, with quadrature/differentiation
checks at $10^{-8}$--$10^{-13}$ --- including a finite-radius probe of the closing-arc
bound \eqref{eq:arcbound}. One sign correction to \cite{KK2024} is
recorded.
\end{enumerate}

\subsubsection*{Relation to the Fox--Barnes $J$-function}
While this work was in preparation, Vaz \cite{Vaz2025} introduced the
\emph{Fox--Barnes $J$-function}, explicitly ``in analogy with the generalization
of the Meijer $G$-function proposed by Karp and Kuznetsov''. Its kernel
\cite[Eq.~(8)]{Vaz2025} is again a ratio of products of double gammas, but each
factor carries its own slope,
$G(b_i+\beta_i s;\tau)$ and $G(1+\tau-a_i-\alpha_i s;\tau)$ with
$\alpha_i,\beta_i>0$, so it stands to the $K$-function as the Fox $H$-function
stands to the Meijer $G$-function; a compensation parameter $\varepsilon$ plays
the role of our $\alpha$. That paper establishes existence conditions, explains
how to select a contour from the parameters, exhibits the Fox $H$-function as a
special case, proves inversion, scaling, modular and differentiation rules, and
computes the Laplace transform of the Kilbas--Saigo function. It is therefore
the natural home of Remark~\ref{rem:foxH} below, and any statement here about a
``double-gamma kernel with arbitrary slopes'' should be read as referring to
\cite{Vaz2025}.

Two points of contact deserve to be stated precisely, since both bear on what
is new here.

\emph{The governing quantity.} In the $J$-function the exponent of the leading
$R^{2}\ln R$ term is controlled not by $N$ but by
$\Delta_{2}=\sum_{i\le m}\beta_i^{2}+\sum_{i\le n}\alpha_i^{2}
-\sum_{i>m}\beta_i^{2}-\sum_{i>n}\alpha_i^{2}$
\cite[Eq.~(15)]{Vaz2025}, which reduces to our $N$ when every slope equals $1$;
Vaz calls $\Delta_{2}\ne0$ \emph{hyper-unbalanced}. (The integer
$N=2(m+n)-(p+q)$ is also defined there \cite[Eq.~(14)]{Vaz2025}, but it enters
only the coefficient of $\ln R$.) The regime $\Delta_{2}<0$ is thus the
$J$-function analogue of the case treated below. Vaz records
\cite[\S3.1(i).2]{Vaz2025} that in it one may take either the left loop
$\mathcal{C}_{-\infty}$ or the right loop $\mathcal{C}_{+\infty}$, and the two
diagonal contours are explicitly set aside \cite[\S3.1(ii)]{Vaz2025}. He then
notes that ``different contours define different functions if they cannot be
continuously deformed into each other'' and, for that reason, adopts the
vertical contour $\mathcal{C}_{ic\infty}$ as the primary definition
\cite[\S3.2]{Vaz2025} --- so that all four clauses of the resulting
vertical-line condition \cite[Def.~1]{Vaz2025}, and every analytic result
resting on it, require $\Delta_{2}\ge0$. The multi-branch structure of the
negative regime is therefore left open there exactly as it is in
\cite{KK2024}; the one trace of it is Example~4 of \cite[\S3.2]{Vaz2025},
where the author observes that only one of the two loops gives a non-null
function --- an instance of our Theorem~\ref{thm:D}.

\emph{The transform calculus.} Here the overlap is real and we make no priority
claim. The two multiplier identities behind our insertion rules
(Lemma~\ref{lem:insert}) already appear, used \emph{ad hoc}, in
\cite[Eqs.~(36)--(38) and proof of Prop.~7]{Vaz2025}, and two of the operators
we treat are there in greater slope generality: differentiation
\cite[Prop.~5]{Vaz2025} --- with the same prefactor $\tau^{k}$ and the same
enlargement $(p,q)\mapsto(p+2k,q+2k)$ --- and the Laplace transform
\cite[Prop.~7]{Vaz2025}. What we add on this side is the organisation of the
rules into a closed calculus, the Erd\'elyi--Kober, Euler and Mellin-convolution
cases, and above all the invariance of $N$ under every rule
(Theorem~\ref{thm:Ninvariant}) with its consequence, the ``$N=0$ trap'' of
Corollary~\ref{cor:Ntrap}. That corollary is independently corroborated by
\cite[Eq.~(32)]{Vaz2025}: reducing a $J$-function to a Fox $H$-function forces
$\Delta_{2}=\Theta_{2}=\Omega_{2}=N=0$.

The present paper is thus complementary rather than more general. We stay inside
the Karp--Kuznetsov kernel --- one common $\tau$, all slopes $\pm1$ --- and
address the one regime that \cite{KK2024} excludes, where the contour is
\emph{not} determined by the parameters: the classification of the admissible
contours, the relations among the resulting functions, the residue expansions
and their recurrence, and the invariance of $N$. Those are the contributions
claimed here.

\subsection{Relation to classical Mellin--Barnes theory}
For the Meijer $G$- and Fox $H$-functions the analogous phenomenon is classical:
when the vertical-line integral diverges one uses loop contours
$\mathcal{L}_{+\infty}$ or $\mathcal{L}_{-\infty}$, and these give genuinely
different functions with different existence conditions
\cite{MathaiSaxena1973,KilbasSaigo2004,ParisKaminski2001}. The $N<0$ $K$-function
world is however \emph{richer} than the classical one. Because $G(\cdot;\tau)$
has order two, the decay in the horizontal sectors is of order
$e^{-cr^{2}\ln r}$ rather than $e^{-c|t|}$; consequently all four contours
converge simultaneously and for all $z$, so one really obtains four globally
defined functions rather than analytic continuations of one another across
$|z|=1$.

\section{Preliminaries}\label{sec:prelim}

Throughout, $\tau>0$; $\C^{\pm}$ are the open upper/lower half-planes;
$\log$ is the principal branch, and $z$ ranges over the Riemann surface of the
logarithm unless stated otherwise. For $\eps\in(0,\pi)$ put
$S^{\pm}_{\eps}=\{s\ne0:\eps\le\pm\arg s\le\pi-\eps\}$.

\subsection{The Barnes double gamma function}
$G(z;\tau)$ is the entire function of $z$ of order two given by the Weierstrass
product \cite[Eq.~(3)]{KK2024}, normalised by $G(1;\tau)=1$ and satisfying
\eqref{eq:quasiper}. Its zeros are the lattice points
\begin{equation}\label{eq:zeros}
z=-(l\tau+k),\qquad l,k\in\Z_{\ge0},
\end{equation}
and they are all simple if and only if $\tau\notin\Q$. We shall use the
modular transformation \cite[Eq.~(5)]{KK2024}
\begin{equation}\label{eq:modular}
G(z;\tau)=(2\pi)^{\frac{z}{2}\left(1-\frac1\tau\right)}\,
\tau^{\frac{z-z^{2}}{2\tau}+\frac{z}{2}-1}\,
G\!\left(\tfrac{z}{\tau};\tfrac1\tau\right),
\end{equation}
and the representation of $\Gamma$ that embeds Meijer $G$ into $K$,
\begin{equation}\label{eq:GammaviaG}
\Gamma(s)=(2\pi)^{\frac{1-\tau}{2}}\tau^{s-\frac12}\,
\frac{G(s+\tau;\tau)}{G(s;\tau)} .
\end{equation}

Two elementary facts will be used repeatedly. The first is immediate from the
Weierstrass product, in which the only factor vanishing at $z=0$ is the
prefactor $z/\tau$.

\begin{lemma}\label{lem:Gzero}
$G(z;\tau)=\dfrac{z}{\tau}+O(z^{2})$ as $z\to0$; in particular
$G'(0;\tau)=1/\tau$.
\end{lemma}

\begin{lemma}\label{lem:Gprime}
Let $\tau\notin\Q$ and $l,k\in\Z_{\ge0}$, and set $\rho=l\tau+k$. Then the zero
of $G(\cdot;\tau)$ at $-\rho$ is simple and
\begin{equation}\label{eq:Gprime}
G'(-\rho;\tau)=\frac{1}{\tau\,D_{\tau}(l,k)},\qquad
D_{\tau}(l,k)=\prod_{j=0}^{k-1}\Gamma\!\left(\frac{j-k}{\tau}-l\right)
\prod_{i=0}^{l-1}(2\pi)^{\frac{\tau-1}{2}}\tau^{\frac12+(l-i)\tau}\,
\Gamma\big((i-l)\tau\big).
\end{equation}
\end{lemma}

\begin{proof}
Iterating the first relation in \eqref{eq:quasiper} $k$ times and then the second
$l$ times gives $G(z+\rho;\tau)=D_{\tau}(l,k;z)\,G(z;\tau)$ with
\[
D_{\tau}(l,k;z)=\prod_{j=0}^{k-1}\Gamma\!\left(\frac{z+j}{\tau}\right)
\prod_{i=0}^{l-1}(2\pi)^{\frac{\tau-1}{2}}\tau^{\frac12-z-k-i\tau}\,
\Gamma(z+k+i\tau).
\]
Put $z=-\rho+u$. The gamma factors are evaluated at
$\tfrac{j-k}{\tau}-l+\tfrac{u}{\tau}$ $(0\le j\le k-1)$ and at
$(i-l)\tau+u$ $(0\le i\le l-1)$; since $\tau\notin\Q$ and $j<k$, $i<l$, none of
these arguments is a non-positive integer at $u=0$, so
$D_{\tau}(l,k):=D_{\tau}(l,k;-\rho)$ is finite and non-zero. Therefore
$G(-\rho+u;\tau)=G(u;\tau)/D_{\tau}(l,k;-\rho+u)$, and Lemma~\ref{lem:Gzero}
gives $G(-\rho+u;\tau)=u/(\tau D_{\tau}(l,k))+O(u^{2})$.
\end{proof}

\subsection{A Binet representation with explicit coefficients}
Numerical work requires evaluating $G(z;\tau)$ for general $\tau$, for which no
implementation exists. The starting point is the Binet-type representation
\cite[Prop.~1]{KK2024}, but that proposition prints only three of its six
coefficients. We record all of them.

\begin{lemma}\label{lem:binet}
For $\Rea z>0$ and $\tau>0$,
\begin{equation}\label{eq:binet}
\ln G(z;\tau)=P_{\tau}(z)\ln z+Q_{\tau}(z)-\int_{0}^{\infty}e^{-zx}f_{3}(x;\tau)\,dx,
\end{equation}
where, with $f(x)=\dfrac{x^{2}}{(1-e^{-x})(1-e^{-\tau x})}$ and
$f_{3}(x)=x^{-3}\big(f(x)-f(0)-f'(0)x-\tfrac12 f''(0)x^{2}\big)$,
\begin{align}
P_{\tau}(z)&=\frac{z^{2}}{2\tau}-\frac{1+\tau}{2\tau}\,z
            +\frac{1+3\tau+\tau^{2}}{12\tau}
            \;=\;\tfrac12 B_{2,2}(z;1,\tau),\label{eq:P}\\
Q_{\tau}(z)&=-\frac{\tfrac32+\ln\tau}{2\tau}\,z^{2}
   +\left(\frac{\ln(2\pi\tau)}{2}+\frac{1+\tau+\ln\tau}{2\tau}\right)z+b_{0}(\tau),
   \label{eq:Q}
\end{align}
and $b_{0}(\tau)$ is determined by $G(1;\tau)=1$. In particular
$b_{0}(1)=\zeta'(-1)-\tfrac12\ln 2\pi$.
\end{lemma}

\begin{proof}
The point of the proof is \emph{not} to guess a global identity from an
asymptotic match. Existence of a representation of the form \eqref{eq:binet}
with \emph{some} polynomials $P_\tau,Q_\tau$ of degree at most two is already
contained in \cite[Prop.~1]{KK2024}, which in turn rests on the Binet-type
formula for $\ln\Gamma_2$ in \cite{Ruijsenaars2000,Spreafico2009}. What is left
is to identify the six coefficients, and for that the asymptotic match is
legitimate: two polynomials of degree $\le2$ that agree to order $O(1/z)$
coincide.

So write $\Lambda(z)$ for the right-hand side of \eqref{eq:binet} with the
still-unknown $P_{\tau}(z)=a_2z^2+a_1z+a_0$ and $Q_\tau(z)=b_2z^2+b_1z+b_0$.
Since $f_3$ is
analytic near $0$ and bounded on $[0,\infty)$ up to a $O(1/x)$ tail, Watson's
lemma gives $\int_0^\infty e^{-zx}f_3(x)dx=O(1/z)$ as $z\to+\infty$. Substituting
$\Lambda$ into the first relation of \eqref{eq:quasiper} and expanding as
$z\to+\infty$,
\[
\Lambda(z+1)-\Lambda(z)=2a_2z\ln z+(a_2+a_1)\ln z+(a_2+2b_2)z
   +\big(\tfrac32a_2+a_1+b_2+b_1\big)+O(1/z),
\]
whereas
$\ln\Gamma(z/\tau)=\tfrac{z}{\tau}\ln z-\tfrac{z}{\tau}\ln\tau-\tfrac12\ln z
+\tfrac12\ln\tau-\tfrac{z}{\tau}+\tfrac12\ln2\pi+O(1/z)$. Matching the
coefficients of $z\ln z$, $\ln z$, $z$ and $1$ in turn yields
$a_2=\tfrac1{2\tau}$, $a_1=-\tfrac{1+\tau}{2\tau}$,
$b_2=-\tfrac{3/2+\ln\tau}{2\tau}$ and the stated $b_1$. The same computation
applied to the second relation of \eqref{eq:quasiper} reproduces these four
values identically, so the two functional equations are consistent and neither
determines $a_0$ or $b_0$ (both enter only at order $O(1/z)$ or as an additive
constant). The constant $a_0$ is fixed by comparison with the complete
asymptotic expansion of $\ln G$ \cite{BillinghamKing1997,AlexanianKuznetsov2023},
whose $\ln z$-coefficient is $\tfrac12 B_{2,2}(z;1,\tau)$, giving
$a_0=\tfrac{1+3\tau+\tau^{2}}{12\tau}$; at $\tau=1$ this is $\tfrac5{12}$,
consistent with the classical Barnes-$G$ expansion. (This is the one coefficient
we take from the literature rather than derive; it is confirmed independently to
$22$ digits in \S\ref{sec:numerics}.) Finally $b_0$ is fixed by
$\Lambda(1)=\ln G(1;\tau)=0$, i.e.\ $b_{0}=\int_0^\infty e^{-x}f_3(x)dx-b_2-b_1$.
At $\tau=1$ the classical expansion of $\ln G(z+1)$ has constant term
$\zeta'(-1)$, whence $b_0(1)=\zeta'(-1)-\tfrac12\ln2\pi$.
\end{proof}

\begin{remark}
Both $a_0$ and $b_0(1)$ are easy to get wrong: the frequently quoted constant
$\zeta'(-1)+\tfrac1{12}$ belongs to a different normalisation of the Barnes
$G$-function. The values in Lemma~\ref{lem:binet} are confirmed to $22$ digits
against \texttt{mpmath.barnesg} in \S\ref{sec:numerics}.
\end{remark}

\subsection{The auxiliary parameters}
Following \cite[Eq.~(14)]{KK2024} we set, besides $N$ in \eqref{eq:Nmunuxi},
\begin{equation}\label{eq:munuxi}
\nu=\sum_{j=1}^{p}a_j-\sum_{j=1}^{q}b_j,\quad
\mu=\sum_{j=1}^{n}a_j-\sum_{j=n+1}^{p}a_j+\sum_{j=1}^{m}b_j-\sum_{j=m+1}^{q}b_j,
\end{equation}
and $\xi$ defined as $\mu$ with each parameter squared. We also write
\begin{equation}\label{eq:lambdas}
\lambda_{\sharp}=\max_{m<j\le q}\Rea b_j-\tau-1,\qquad
\lambda_{\flat}=\min_{n<j\le p}\Rea a_j,
\end{equation}
with the conventions $\lambda_\sharp=-\infty$ if $m=q$ and
$\lambda_\flat=+\infty$ if $n=p$.

\section{\texorpdfstring{The geometry of the case $N<0$}{The geometry of the case N<0}}\label{sec:geometry}

\subsection{The pole sets are horizontal combs}
Since $G(\cdot;\tau)$ is entire, the poles of $\Gs$ come only from zeros of the
denominator factors. By \eqref{eq:zeros},
$G(a_j-s;\tau)=0$ iff $s\in a_j+\{l\tau+k\}$ and
$G(1+\tau-b_j+s;\tau)=0$ iff $s\in b_j-\{(l{+}1)\tau+(k{+}1)\}$. This is
\cite[Eq.~(13)]{KK2024}:
\begin{equation}\label{eq:Lambdas}
\Lp=\{a_j+l\tau+k:\ n<j\le p,\ l,k\ge0\},\qquad
\Lm=\{b_j-l\tau-k:\ m<j\le q,\ l,k\ge1\},
\end{equation}
and $\Gs$ is meromorphic with \emph{pole set contained in} $\Lp\cup\Lm$. The
inclusion can be strict: a numerator factor may vanish at a point of
$\Lp\cup\Lm$ and cancel the pole. This is not a pathology --- \S\ref{sec:closure}
shows that the transform calculus \emph{always} produces such cancellations
(Remark~\ref{rem:degenerate}) --- so we state it as an inclusion throughout and
impose \eqref{eq:generic} whenever simplicity of the poles is actually used.

The following observation is elementary but is what makes the whole
classification finite; it does not appear in \cite{KK2024}.

\begin{proposition}\label{prop:combs}
Let $\tau>0$. Then $\Lp$ is contained in the $p-n$ horizontal lines
$\Ima s=\Ima a_j$ and satisfies $\Rea s\ge\lambda_\flat$ with
$\Rea s\to+\infty$; and $\Lm$ is contained in the $q-m$ horizontal lines
$\Ima s=\Ima b_j$ and satisfies $\Rea s\le\lambda_\sharp$ with
$\Rea s\to-\infty$.
\end{proposition}

\begin{proof}
$l\tau+k\in\R_{\ge0}$ for $\tau>0$ and $l,k\in\Z_{\ge0}$, so translating $a_j$
(resp.\ $b_j$) by $\pm(l\tau+k)$ never changes the imaginary part. The stated
real-part bounds are \eqref{eq:lambdas}.
\end{proof}

Thus $\Lp$ is a comb running east and $\Lm$ a comb running west, each supported
on finitely many horizontal lines. In particular $\Lp\cap\Lm=\varnothing$ holds
automatically when $\lambda_\sharp<\lambda_\flat$, which we assume throughout;
this is the same non-degeneracy hypothesis as \cite[Def.~1]{KK2024}.

\subsection{Decay sectors}
Recall \cite[Prop.~4(i)]{KK2024}: if $N\ne0$ then, uniformly as $s\to\infty$ in
$S^{\pm}_{\eps}$,
\begin{equation}\label{eq:Gasym}
\Gs(s)=\exp\!\left(\frac{N}{2\tau}s^{2}\ln s+O(|s|^{2})\right).
\end{equation}
Writing $s=re^{i\theta}$ and using
$\Rea(s^{2}\ln s)=r^{2}\big[\cos(2\theta)\ln r-\theta\sin(2\theta)\big]$, the
term $r^{2}\ln r$ dominates both the $O(r^{2})$ remainder, the factor
$|e^{\pi\alpha s^{2}/\tau}|=e^{O(r^{2})}$ and the factor $|z^{-s}|=e^{O(r)}$.
Hence:

\begin{proposition}\label{prop:sectors}
Let $N\ne0$ and let $\theta$ be fixed with $\eps\le|\theta|\le\pi-\eps$. Then
$|\phi(s)z^{-s}|\to0$ super-exponentially along $\arg s=\theta$ if and only if
$N\cos2\theta<0$. Consequently
\begin{itemize}[leftmargin=1.6em,itemsep=1pt]
\item if $N>0$ the decay sectors are the two \emph{vertical} ones,
$|\theta\mp\tfrac\pi2|<\tfrac\pi4$;
\item if $N<0$ the decay sectors are the two \emph{horizontal} ones,
$|\theta|<\tfrac\pi4$ and $|\theta-\pi|<\tfrac\pi4$.
\end{itemize}
\end{proposition}

\subsection{Admissible contours, and why there are exactly four classes}
``Four'' is a statement about a precisely delimited set of contours; without the
delimitation it is false, since one may always adjoin a small loop around a
finite set of poles and change the value by a residue. We therefore fix the
following.

\begin{definition}[Admissible contour]\label{def:admissible}
Let $P=\Lp\cup\Lm$ and assume $\Lp\cap\Lm=\varnothing$. A contour $\gamma$ is
\emph{admissible} if
\begin{enumerate}[label=\textup{(A\arabic*)},leftmargin=2.6em,itemsep=1pt]
\item $\gamma$ is a simple, non-self-intersecting, piecewise-smooth curve
      $\R\to\C\setminus P$, proper (it leaves every compact set at both ends);
\item outside some disc $\gamma$ coincides with two straight rays, of angles
      $\theta^-\in(-\pi,0)$ and $\theta^+\in(0,\pi)$, and $\gamma$ is traversed
      from the $\theta^-$ end to the $\theta^+$ end;
\item $\C\setminus\gamma$ has two components; $\Lp$ lies in the one containing
      the large positive reals and $\Lm$ in the other (equivalently, $\gamma$
      leaves $\Lm$ on its left and $\Lp$ on its right);
\item $\int_\gamma|\phi(s)z^{-s}||ds|<\infty$.
\end{enumerate}
We write $L_{\theta^-,\theta^+}$ for such a contour with the indicated angles.
Two admissible contours are \emph{equivalent} if one can be deformed into the
other through admissible contours; the value of \eqref{eq:Kdef} is then the same.
\end{definition}

Condition (A1) is what excludes extra loops; (A3) is what excludes moving a pole
from one side to the other. Under (A4), Theorem~\ref{thm:A} shows that only the
pair $(\theta^-,\theta^+)$ matters, and only through the decay sector each angle
lies in. The classification below is therefore a classification of
\emph{end-sector types}; that each type is a \emph{single} equivalence class is
not automatic, and is the content of Lemma~\ref{lem:isotopy}.

If $N>0$, both $\theta^{\pm}$ must lie in the vertical decay sectors, i.e.\
$\theta^{+}\in(\tfrac\pi4,\tfrac{3\pi}4)$ and
$\theta^{-}\in(-\tfrac{3\pi}4,-\tfrac\pi4)$. Both intervals are connected, so
any two admissible contours are homotopic through the decay region: the function
is unique. This is why case (i) of \cite{KK2024} is clean.

If $N<0$, each end must lie in one of the \emph{two} horizontal sectors, and the
two sectors are separated by the vertical growth sectors, through which no
homotopy can pass, because (A4) fails on any intermediate ray. Counting end
sectors gives four possibilities; the following lemma is what turns that count
into a count of equivalence classes.

Two things have to be checked, and they are of different kinds: a topological
statement, that contours of the same type can be deformed into one another; and
an analytic one, that the value of \eqref{eq:Kdef} does not change along such a
deformation. We separate them.

\begin{lemma}[Normal form and invariance]\label{lem:isotopy}
Let $N<0$, let $\lambda_{\sharp}<\lambda_{\flat}$ and fix
$\sigma_{0}\in(\lambda_{\sharp},\lambda_{\flat})$.
\begin{enumerate}[label=\textup{(\roman*)},leftmargin=2.4em,itemsep=2pt]
\item Every admissible contour is properly ambient-isotopic, inside
$\C\setminus P$ and through contours satisfying \textup{(A1)}--\textup{(A3)},
to a \emph{canonical} contour obtained from
$\{\sigma_{0}+re^{i\theta^{-}}:r\ge0\}\cup
 \{\sigma_{0}+re^{i\theta^{+}}:r\ge0\}$ with the same terminal angles,
with arbitrarily small local detours inserted if one of these rays meets a point
of $P$.  Each ray meets each horizontal support line at most once, so only
finitely many such detours are required.
\item Along any such isotopy the value of \eqref{eq:Kdef} is constant. Hence the
equivalence classes of Definition~\ref{def:admissible} are in bijection with the
four ordered pairs of terminal sectors.
\end{enumerate}
\end{lemma}

\begin{proof}
\emph{(i) is topology.} Three standard inputs are used. First, a proper simple
arc $\gamma\subset\C$ becomes, after adding the point at infinity, a Jordan
curve in $S^{2}=\C\cup\{\infty\}$; by the Jordan curve theorem it separates
$S^{2}$ into two discs, and by the Schoenflies theorem there is a homeomorphism
of $S^{2}$ carrying $\gamma\cup\{\infty\}$ to a round circle
\cite[Ch.~V--VI]{Newman1951}, \cite[Ch.~9]{Moise1977}. This is what makes the
phrases ``two components'' in \textup{(A3)} and ``left'' and ``right''
meaningful, and it is the only place where we use them. Second, by
Proposition~\ref{prop:combs} the open strip
$V=\{\lambda_{\sharp}<\Rea s<\lambda_{\flat}\}$ meets neither $\Lp$ nor $\Lm$,
and every admissible contour crosses $V$ by \textup{(A3)}; so
$\C\setminus P$ contains a pole-free channel joining the two ends. Third, exhaust $\C$ by closed discs $D_j$ whose boundaries avoid $P$ and
meet the two terminal rays transversely. For $j$ large, the truncation of an
admissible contour in $D_j$ is a cross-cut of the finitely punctured disc
$D_j\setminus(P\cap D_j)$. Condition \textup{(A3)} says more than merely
specifying its two ends: it fixes exactly which punctures lie on each side of
the cross-cut. Put two such cross-cuts in transverse position. If they intersect,
an innermost bigon is bounded by one subarc of each. That bigon cannot contain a
puncture, since a puncture inside it would lie on different sides of the two
cross-cuts, contrary to the common partition imposed by \textup{(A3)}. Removing
empty bigons therefore makes the cross-cuts disjoint. The closed strip between
them contains no puncture, and the Jordan--Schoenflies theorem slides one across
that strip to the other, relative to the boundary and $P\cap D_j$; this is the
arc form of the homotopy--isotopy principle
\cite[\S4]{Epstein1966}, \cite[\S1.2]{FarbMargalit2012}. Applying isotopy
extension on the nested discs and choosing the isotopies compatibly gives, by
exhaustion, a proper ambient isotopy in $\C\setminus P$. Inside the pole-free channel $V$, empty-bigon isotopies remove any remaining
backtracking; the resulting single crossing can then be moved to $\sigma_0$,
while the terminal rays are rotated inside their connected decay sectors. Thus the proper isotopy class is determined by
the ordered pair of terminal sectors \emph{together with the fixed separation
of $\Lm$ and $\Lp$ imposed by \textup{(A3)}}. Conversely, contours with
different terminal sectors are not equivalent, because rotating an end between
the two horizontal sectors necessarily crosses a vertical growth sector, where
\textup{(A4)} fails by Proposition~\ref{prop:sectors}.

\emph{(ii) is analysis, and uses no closing-arc estimate near the real axis.}
Let $\gamma_{0},\gamma_{1}$ be two members of an admissible isotopy that agree
outside a disc up to a bounded offset of their terminal rays, and that are close
enough that the region between them is contained in a simply connected subset of
$\C\setminus P$. Truncate both at $|s|=R$ and join the four endpoints by the two
short circular arcs $\alpha^{\pm}_{R}$. The resulting closed chain bounds a
region free of poles, so Cauchy's theorem gives
\[
\int_{\gamma_{0}(R)}\phi(s)z^{-s}ds-\int_{\gamma_{1}(R)}\phi(s)z^{-s}ds
=\pm\int_{\alpha^{+}_{R}}\!\!-\int_{\alpha^{-}_{R}} .
\]
Each $\alpha^{\pm}_{R}$ sweeps an angular interval whose closure lies in
\emph{one} terminal sector, hence at angular distance at least some
$\eps'>0$ from $\arg s\in\{0,\pi\}$; on it \eqref{eq:Gasym} is available and
Proposition~\ref{prop:sectors} gives
$|\phi(s)z^{-s}|\le\exp(-c_{\eps'}R^{2}\ln R)$, while the arcs have length at most
$O(R)$. This polynomial factor is dominated by the super-exponential decay, so
letting $R\to\infty$ the two integrals agree. A finite chain of such
steps, extracted from the isotopy by compactness of the parameter interval,
covers the whole deformation. This is the reason the deformation invariance
asserted in Definition~\ref{def:admissible} is legitimate; note in particular
that no arc crossing a neighbourhood of the real axis occurs here, so
\S\ref{sec:arcs} is not invoked.
\end{proof}

\begin{remark}
Only the terminal behaviour is topologically relevant, which is why the
classification is finite even though $\C\setminus P$ is an infinite-type
planar surface with infinitely many punctures (and genus zero): $P$ is discrete,
and condition \textup{(A3)} fixes the puncture partition throughout the
exhaustion used above.
\end{remark}

There are therefore exactly $2\times2=4$ equivalence classes of admissible
contours.

Throughout the rest of the paper we fix two angles
\begin{equation}\label{eq:epsconv}
0<\eps_{0}<\eps_{1}<\tfrac{\pi}{4},
\end{equation}
using $\eps_{0}$ for the terminal rays of a contour and reserving the larger
$\eps_{1}$ for the closing arcs of \S\ref{sec:arcs}. The margin
$\eps_{1}-\eps_{0}$ is what lets a ray whose vertex is $\sigma_{0}\ne0$ be closed
by an arc centred at the origin; see the proof of Theorem~\ref{thm:B}.

\begin{definition}\label{def:four}
Let $N<0$ and $\lambda_\sharp<\lambda_\flat$. With $\eps=\eps_{0}$ as in
\eqref{eq:epsconv} define
\[
\begin{array}{llll}
\Kr: & \gamma=L_{-\eps,\,\eps}, & \text{both ends east,} &
   \text{encircles }\Lp\text{ clockwise};\\[2pt]
\Kl: & \gamma=L_{-\pi+\eps,\,\pi-\eps}, & \text{both ends west,} &
   \text{encircles }\Lm\text{ counterclockwise};\\[2pt]
\Ku: & \gamma=L_{-\eps,\,\pi-\eps}, & \text{SE}\to\text{NW};\\[2pt]
\Kd: & \gamma=L_{-\pi+\eps,\,\eps}, & \text{SW}\to\text{NE},
\end{array}
\]
each by the integral \eqref{eq:Kdef} over the indicated contour.
\end{definition}

Figure~\ref{fig:contours} shows the four classes together with the sector
structure.

\begin{figure}[t]
\centering
\begin{tikzpicture}[scale=0.92,>=Stealth,
    pole/.style={circle,inner sep=0.9pt},
    pp/.style={pole,fill=red!70!black},
    pm/.style={pole,fill=blue!60!black}]
\begin{scope}
  \fill[black!8] (0,0) -- (45:2.4) arc[start angle=45,end angle=135,radius=2.4] -- cycle;
  \fill[black!8] (0,0) -- (-45:2.4) arc[start angle=-45,end angle=-135,radius=2.4] -- cycle;
  \draw[gray!45] (0,0) circle (2.4);
  \draw[gray!50,dashed] (-135:2.4)--(45:2.4);
  \draw[gray!50,dashed] (135:2.4)--(-45:2.4);
  \draw[->,gray!65] (-2.75,0) -- (2.8,0) node[right,font=\tiny]{$\Rea s$};
  \draw[->,gray!65] (0,-2.75) -- (0,2.8) node[above,font=\tiny]{$\Ima s$};
  \foreach \x in {0.6,1.05,1.5,1.95,2.35} \node[pp] at (\x,0) {};
  \foreach \x in {-0.75,-1.2,-1.65,-2.1,-2.45} \node[pm] at (\x,0) {};
  \draw[very thick,black,postaction={decorate,decoration={markings,
        mark=at position 0.62 with {\arrow{Stealth}}}}] (-0.1,-2.75) -- (-0.1,2.75);
  \node[font=\tiny,gray!50!black] at (0,1.55) {decay};
  \node[font=\tiny,gray!50!black] at (0,-1.55) {decay};
  \node[font=\tiny,gray!50!black] at (1.6,0.55) {grow};
  \node[font=\tiny,gray!50!black] at (-1.6,0.55) {grow};
  \node[font=\scriptsize] at (0,-3.25) {(a) $N>0$: one contour class};
\end{scope}
\begin{scope}[xshift=7.6cm]
  \fill[black!8] (0,0) -- (-45:2.4) arc[start angle=-45,end angle=45,radius=2.4] -- cycle;
  \fill[black!8] (0,0) -- (135:2.4) arc[start angle=135,end angle=225,radius=2.4] -- cycle;
  \draw[gray!45] (0,0) circle (2.4);
  \draw[gray!50,dashed] (-135:2.4)--(45:2.4);
  \draw[gray!50,dashed] (135:2.4)--(-45:2.4);
  \draw[->,gray!65] (-2.75,0) -- (2.8,0) node[right,font=\tiny]{$\Rea s$};
  \draw[->,gray!65] (0,-2.75) -- (0,2.8) node[above,font=\tiny]{$\Ima s$};
  \foreach \x in {0.6,1.05,1.5,1.95,2.35,2.7} \node[pp] at (\x,0) {};
  \foreach \x in {-0.75,-1.2,-1.65,-2.1,-2.45,-2.8} \node[pm] at (\x,0) {};
  \node[font=\tiny,gray!50!black] at (0,1.7) {grow};
  \node[font=\tiny,gray!50!black] at (0,-1.7) {grow};
  \draw[very thick,red!70!black,postaction={decorate,decoration={markings,
        mark=at position 0.30 with {\arrow{Stealth}},
        mark=at position 0.80 with {\arrow{Stealth}}}}]
        (3.05,-0.45) -- (0.05,-0.45) arc[start angle=-90,end angle=90,radius=0.45] -- (3.05,0.45);
  \draw[very thick,blue!60!black,postaction={decorate,decoration={markings,
        mark=at position 0.30 with {\arrow{Stealth}},
        mark=at position 0.80 with {\arrow{Stealth}}}}]
        (-3.05,-0.95) -- (-0.05,-0.95) arc[start angle=270,end angle=90,radius=0.95] -- (-3.05,0.95);
  \draw[very thick,green!35!black,dash pattern=on 3pt off 2pt,
        postaction={decorate,decoration={markings,
        mark=at position 0.5 with {\arrow{Stealth}}}}]
        (3.05,-1.55) .. controls (0.9,-1.55) and (-0.9,1.55) .. (-3.05,1.55);
  \node[font=\scriptsize,red!70!black]   at (2.35,0.78) {$\Kr$};
  \node[font=\scriptsize,blue!60!black]  at (-2.4,1.3)  {$\Kl$};
  \node[font=\scriptsize,green!35!black] at (-1.15,1.95) {$\Ku$};
  \node[font=\scriptsize] at (0,-3.25) {(b) $N<0$: four contour classes};
\end{scope}
\end{tikzpicture}
\caption{Shaded: the sectors in which $\phi(s)z^{-s}$ decays
super-exponentially (Proposition~\ref{prop:sectors}). Red dots: the poles $\Lp$
(a comb running east); blue dots: $\Lm$ (running west). (a) For $N>0$ both
ends of the contour are trapped in the vertical sectors, and all admissible
contours are homotopic. (b) For $N<0$ each end may go east or west, giving four
non-homotopic classes; $\Kd$, omitted for clarity, is the reflection of $\Ku$
in the real axis.}
\label{fig:contours}
\end{figure}
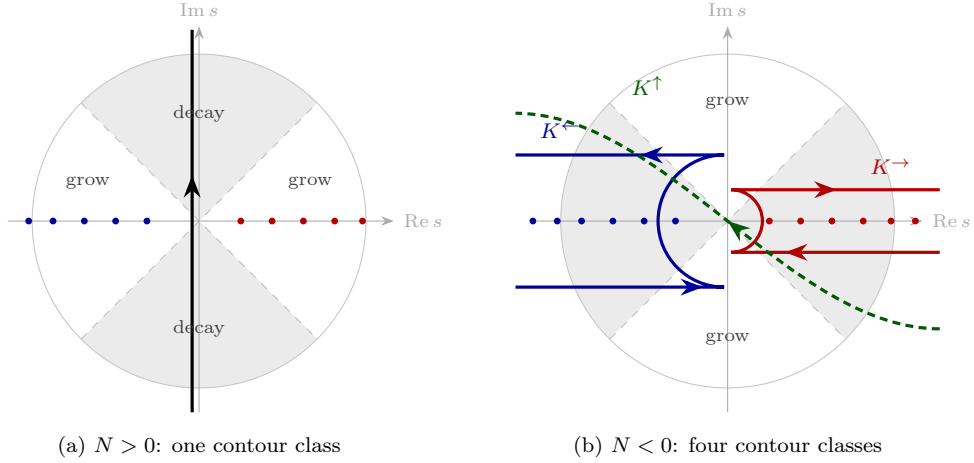

\section{Unconditional convergence}

\begin{theorem}[Convergence and analyticity]\label{thm:A}
Let $N<0$, $\tau>0$, $\lambda_\sharp<\lambda_\flat$, and let $\alpha\in\C$ be
\emph{arbitrary}. Let $\gamma=L_{\theta^-,\theta^+}$ be any contour whose two
terminal rays satisfy $\cos2\theta^{\pm}>0$ and
$\eps\le|\theta^{\pm}|,\,|\pi-|\theta^{\pm}||$. Then along $\gamma$
\begin{equation}\label{eq:bound}
\big|\phi(s)z^{-s}\big|
\le\exp\!\left(-\frac{|N|}{2\tau}\cos(2\theta)\,r^{2}\ln r\,(1+o(1))\right),
\qquad r=|s|\to\infty,
\end{equation}
so the integral \eqref{eq:Kdef} converges absolutely for every $z$ on the
Riemann surface of the logarithm; its value depends only on the homotopy class
of $\gamma$; and $w\mapsto K(e^{2\pi w})$ extends to an entire function of
$w\in\C$.
\end{theorem}

\begin{proof}
The bound \eqref{eq:bound} is Proposition~\ref{prop:sectors} together with
\eqref{eq:Gasym}, valid uniformly in $S^{\pm}_{\eps}$. Absolute convergence and
$\gamma$-independence within a class follow by rotating the terminal rays inside
the decay sector. Since \eqref{eq:bound} holds uniformly for $\arg z$ in any
compact set and the bound is independent of $z$ up to the factor
$e^{O(r)}$, differentiation under the integral sign is legitimate for all
$w=\tfrac{1}{2\pi}\log z\in\C$.
\end{proof}

\begin{remark}[What disappears]\label{rem:disappears}
Theorem~\ref{thm:A} should be compared with the list of parameters excluded in
\cite[p.~235]{KK2024}. For $N<0$ \emph{none} of the conditions
$\Rea\alpha\ge0$, $\Rea\nu<0$, $4\alpha\ne\pm(p-q)i$, $\arg\mu\ne\pm\pi/2$ is
needed; the seven-case analysis and Table~1 of \cite{KK2024} are vacuous here.
The entire difficulty of the case has migrated from ``does the integral
converge'' to ``which of the four is it''.
\end{remark}

\begin{remark}[A caveat for computation]\label{rem:caveat}
Estimate \eqref{eq:Gasym} is stated in \cite{KK2024} only on $S^{\pm}_{\eps}$,
which \emph{excludes} a neighbourhood of the real axis --- exactly where the
pole combs live. Theorem~\ref{thm:A} is therefore proved with genuine rays whose
angles are bounded away from $0$ and $\pi$. For computation it is far more
convenient to realise the same homotopy classes by horizontal lines
$\Ima s=\pm h$; these leave $S^{\pm}_{\eps}$ at large $|s|$, so their use is
justified a posteriori by the $h$-independence of the computed value
(\S\ref{sec:numerics}, T1) rather than by \eqref{eq:Gasym}. See also
\S\ref{sec:pitfalls}.
\end{remark}

\section{Residue expansions}

Throughout this section we assume the genericity hypothesis
\begin{equation}\label{eq:generic}
\tau\notin\Q,\qquad
\begin{cases}
a_i-a_j\notin\Z+\tau\Z, & 1\le i<j\le p,\\
b_i-b_j\notin\Z+\tau\Z, & 1\le i<j\le q,\\
a_i-b_j\notin\Z+\tau\Z, & 1\le i\le p,\ 1\le j\le q.
\end{cases}
\end{equation}
The condition is stated over pairs of \emph{factor indices}, not over the
underlying set of values: two parameters sitting in different factors are subject
to it even when they happen to be numerically equal (indeed $a_i=a_j$ with
$i\ne j$ already violates it, as it must, since then the two zero lattices
coincide). Two separate things are being excluded. Within a denominator family,
\eqref{eq:generic} makes the zeros simple and pairwise distinct. Across
families, it forbids a \emph{numerator} factor from vanishing at a zero of a
denominator factor: since every zero lattice is a translate of
$-(\Z_{\ge0}+\tau\Z_{\ge0})$ by one of the $c$'s, two such lattices meet exactly
when the corresponding difference lies in $\Z+\tau\Z$. Without that, a putative
pole may be a removable singularity, and the pole set of $\Gs$ is only
\emph{contained} in $\Lp\cup\Lm$. Under \eqref{eq:generic} it equals
$\Lp\cup\Lm$ and every pole is simple. The condition is not vacuous bookkeeping:
Remark~\ref{rem:degenerate} shows that the transform calculus of
\S\ref{sec:closure} violates it \emph{by construction}, every insertion creating
a pair $\{c,c+\tau\}$ whose difference is $-\tau$.

\subsection{The closing arcs}\label{sec:arcs}
Turning a hairpin into a residue sum means closing it and letting the closing
piece recede to infinity. That closing piece is an arc, and an arc joining the
two rays of $L_{-\eps,\eps}$ \emph{necessarily crosses a neighbourhood of the
positive real axis} --- precisely the region excluded from
$S^{\pm}_{\eps}$, and precisely where the poles are. Estimate \eqref{eq:Gasym}
therefore does not apply on it, and the required bound has to be produced
separately. We do so in two steps: a reduction that is unconditional, and one
explicitly stated growth hypothesis of Diophantine type.

\begin{lemma}[Reduction to a double-sine product]\label{lem:reduction}
Let $S_{2}(z;\tau)=(2\pi)^{\frac{1+\tau}{2}-z}G(z;\tau)/G(1+\tau-z;\tau)$ be the
double sine function, so that $S_{2}(z;\tau)S_{2}(1+\tau-z;\tau)=1$ and
\begin{equation}\label{eq:reflect}
G(w;\tau)=(2\pi)^{w-\frac{1+\tau}{2}}\,S_{2}(w;\tau)\,G(1+\tau-w;\tau).
\end{equation}
Applying \eqref{eq:reflect} to the factors $G(b_j-s;\tau)$, $j\le m$, and
$G(a_j-s;\tau)$, $j>n$, gives
\begin{equation}\label{eq:split}
\Gs(s)=(2\pi)^{L(s)}\;\Pi(s)\;\widetilde{\Gs}(s),
\qquad
\Pi(s):=\frac{\prod_{j\le m}S_{2}(b_j-s;\tau)}{\prod_{j>n}S_{2}(a_j-s;\tau)},
\end{equation}
where $L$ is affine in $s$ and $\widetilde\Gs$ is a quotient of $m+n$ over
$(q-m)+(p-n)$ factors $G(c+s;\tau)$ with constant $c$. Consequently, uniformly
for $s\to\infty$ in $|\arg s|\le\pi/4$,
\begin{equation}\label{eq:tildeasym}
\ln\big|\widetilde\Gs(s)\big|
=\frac{N}{2\tau}\Rea\!\big(s^{2}\ln s\big)+O(|s|^{2}).
\end{equation}
\end{lemma}

\begin{proof}
Formula \eqref{eq:reflect} is the definition of $S_2$ rearranged; the inversion
$S_2(z)S_2(1+\tau-z)=1$ is immediate from it. Counting factors after the
substitution: the numerator of $\widetilde\Gs$ carries $G(1+\tau-b_j+s;\tau)$
for $j\le m$ and $G(1+\tau-a_j+s;\tau)$ for $j\le n$, the denominator
$G(1+\tau-b_j+s;\tau)$ for $j>m$ and $G(1+\tau-a_j+s;\tau)$ for $j>n$; the
difference of the counts is $2(m+n)-p-q=N$. Every argument is of the form
$c+s$, so $\arg(c+s)=\arg s+o(1)$ and \emph{all} of them stay in
$|\arg|<\pi-\eps'$ when $|\arg s|\le\pi/4$; Proposition~2 of \cite{KK2024},
which is uniform in that sector, then gives \eqref{eq:tildeasym}.
\end{proof}

By the functional equations
\[
S_{2}(z+1;\tau)=\frac{S_{2}(z;\tau)}{2\sin(\pi z/\tau)},\qquad
S_{2}(z+\tau;\tau)=\frac{S_{2}(z;\tau)}{2\sin(\pi z)},
\]
the factor $\Pi$ in
\eqref{eq:split} is a product of reciprocal sines evaluated on the lattice
$\Z+\tau\Z$. Its size is therefore a Diophantine question about $\tau$ and the
parameters, and it is
the \emph{same} question that governs the convergence of the residue series
(Remark~\ref{rem:diophantine}) and, one rung down, the series representation of
the density of the supremum of a stable process
\cite{HubalekKuznetsov2011,Kuznetsov2013,HackmannKuznetsov2013}. We isolate
exactly what we need.

\begin{itemize}[leftmargin=2.6em]
\item[\textup{(H$^{\sup}_{\mathrm{arc}}$)}] There are $\eps_{1}\in(0,\pi/4)$ and a
sequence $R_{1}<R_{2}<\cdots\to\infty$ with
\[
\sup_{\substack{|\theta|\le\eps_{1}\;\text{or}\;|\theta-\pi|\le\eps_{1}}}
\log^{+}|\Pi(R_j e^{i\theta})|
=o\big(R_{j}^{2}\ln R_{j}\big),
\qquad \log^{+}x:=\max\{\log x,0\}.
\]
\item[\textup{(H$^{\mathrm{int}}_{\mathrm{arc}}$)}] There is
$\eps_{1}\in(0,\pi/4)$ such that, with
\begin{equation}\label{eq:Mdef}
\begin{aligned}
M(r)&:=\int_{|\theta|\le\eps_{1}}|\Pi(re^{i\theta})|\,d\theta
      +\int_{|\theta-\pi|\le\eps_{1}}|\Pi(re^{i\theta})|\,d\theta,\\
\log^{+}\!\left(\int_{R}^{2R}M(r)\,dr\right)
&=o\big(R^{2}\ln R\big),\qquad R\to\infty.
\end{aligned}
\end{equation}
\end{itemize}
We write \Harc for the disjunction: \emph{the arc hypothesis holds} if either
\textup{(H$^{\sup}_{\mathrm{arc}}$)} or
\textup{(H$^{\mathrm{int}}_{\mathrm{arc}}$)} does. Every conditional statement
below assumes only \Harc.

Three remarks on the shape of these conditions. First, the exponent
$o(R^{2}\ln R)$ is all the proof of Lemma~\ref{lem:arcs} consumes, because the
term it has to beat is $-\tfrac{|N|\cos2\eps_{1}}{2\tau}R^{2}\ln R$; an earlier
version of this paper assumed the far stronger $O(R\ln R)$, a full power of $R$
more than is needed. Second,
\textup{(H$^{\mathrm{int}}_{\mathrm{arc}}$)} is a bound on an \emph{average} of
an $L^{1}$ norm, not on a supremum, and it does not require a good sequence of
radii to be exhibited --- Lemma~\ref{lem:adaptive} produces one. That matters,
because a circle can pass arbitrarily close to a pole of $\Pi$ while an annulus
cannot notice: a simple pole is area-integrable in two dimensions. Third, neither
form is proved here for any $\tau$.

We stress what \Harc is and is not. It is an explicit growth bound on the
double-sine quotient $\Pi$ of \eqref{eq:split} along a sequence of closing arcs.
It depends not only on $\tau$ but on the full parameter vectors $\va,\vb$,
on which factor sits in which block, and on the chosen sequence $R_{j}$; the
notation is $\Harc$ rather than $(\mathrm{D}_{\tau})$ for exactly that reason.
It is \emph{motivated} by the small-denominator phenomena that appear in the
residue series of related stable-process problems
\cite{HubalekKuznetsov2011,Kuznetsov2013,HackmannKuznetsov2013}, but those
results concern a specific coefficient structure and do not imply \Harc for
general parameters. Finding simple number-theoretic conditions on $\tau$ (and on
$\va,\vb$) that imply \Harc is left open; see item~\ref{op:diophantine} of
\S\ref{sec:open}.

\begin{lemma}[Closing arcs]\label{lem:arcs}
Let $N<0$ and assume \textup{(H$^{\sup}_{\mathrm{arc}}$)}. Then there are $c>0$
and $J$ with
\begin{equation}\label{eq:arcbound}
\sup_{|\theta|\le\eps_{1}}\big|\phi(R_j e^{i\theta})z^{-R_j e^{i\theta}}\big|
\;\le\;\exp\!\big(-c\,R_{j}^{2}\ln R_{j}\big),\qquad j\ge J,
\end{equation}
and the same bound with $s=R_j e^{i\theta}$ and
$|\theta-\pi|\le\eps_{1}$. In particular the
contribution of those arcs to any closed contour tends to $0$ along the sequence
$R_{j}$.
\end{lemma}

\begin{proof}
On the eastern arc, \eqref{eq:split} and \eqref{eq:tildeasym} give
$\ln|\Gs(s)|\le\frac{N}{2\tau}R^{2}\ln R\cos2\theta+O(R^{2})+\ln|\Pi(s)|$.
Since $N<0$ and $\cos2\theta\ge\cos2\eps_{1}>0$ there, and since
$\ln|e^{\pi\alpha s^{2}/\tau}z^{-s}|=O(R^{2})$ while
$\log^{+}|\Pi(s)|=o(R^{2}\ln R)$ by
\textup{(H$^{\sup}_{\mathrm{arc}}$)}, the term $R^{2}\ln R$ dominates
and \eqref{eq:arcbound} follows with any
$c<\frac{|N|\cos2\eps_{1}}{2\tau}$. For the western arcs apply the same
argument after the substitution $s\mapsto-s$, which by the first identity of
\cite[Eq.~(15)]{KK2024} replaces $\Gs^{m,n}_{p,q}(\va;\vb\,|\,\cdot\,)$ by
$\Gs^{n,m}_{q,p}(1+\tau-\vb;1+\tau-\va\,|\,\cdot\,)$ and leaves $N$ unchanged.
\end{proof}

We now show that the integral form does the same job, and in doing so produces
the sequence of radii instead of assuming it. This is the ``adaptive arc''
statement: one does not choose the closing radii in advance, one selects them
from an annulus after the fact.

\begin{lemma}[Adaptive radius selection]\label{lem:adaptive}
Let $N<0$ and assume \eqref{eq:generic}. Then:
\begin{enumerate}[label=\textup{(\roman*)},leftmargin=2.4em,itemsep=2pt]
\item $M(r)<\infty$ for all but countably many $r>0$, and
$\int_{R}^{2R}M(r)\,dr<\infty$ for every $R>0$.
\item If \textup{(H$^{\mathrm{int}}_{\mathrm{arc}}$)} holds then there is a
sequence $R_{j}\to\infty$ along which
\begin{equation}\label{eq:arcbound2}
R_j\int_{|\theta|\le\eps_{1}}
\big|\phi(R_j e^{i\theta})z^{-R_j e^{i\theta}}\big|\,d\theta
\;\longrightarrow\;0 ,
\end{equation}
and the same with $|\theta-\pi|\le\eps_{1}$. Consequently the conclusion of
Lemma~\ref{lem:arcs} --- that the arc contribution to a closed contour vanishes
along a sequence of radii --- holds, and with it Theorems~\ref{thm:B} and
\ref{thm:D}.
\end{enumerate}
\end{lemma}

\begin{proof}
(i) $\Pi$ is a finite product of double sines and their reciprocals, hence
meromorphic in $\C$; under \eqref{eq:generic} its poles are simple and form a
discrete set $Z$. Only countably many circles $|s|=r$ meet $Z$, and off those
$M(r)$ is a continuous function on a compact angular interval, hence finite. For
the second claim, write the annular sector integral in polar coordinates:
\[
\int_{R}^{2R}\!\!M(r)\,dr\le\frac1R\int_{R}^{2R}\!\!\int_{\text{arcs}}
|\Pi(re^{i\theta})|\,r\,d\theta\,dr
=\frac1R\iint_{A(R)}|\Pi|\,dA ,
\]
$A(R)$ being the union of the two annular sectors. Near a simple pole $s_{0}$ one
has $|\Pi(s)|\le C_{s_{0}}|s-s_{0}|^{-1}$, and $|s-s_{0}|^{-1}$ is locally
integrable with respect to \emph{area}; since $A(R)$ is compact and contains
finitely many poles, the area integral is finite.

(ii) By (i) and the mean value inequality there is, for each $R$, a radius
$r_{R}\in[R,2R]$ with
$M(r_{R})\le\frac1R\int_{R}^{2R}M(r)dr\le\exp\big(o(R^{2}\ln R)\big)$. On the
eastern arc of radius $r_{R}$, \eqref{eq:split} and \eqref{eq:tildeasym} give
\[
\big|\phi(s)z^{-s}\big|\le|\Pi(s)|\,
\exp\!\Big(-\tfrac{|N|\cos2\eps_{1}}{2\tau}r_{R}^{2}\ln r_{R}+O(r_{R}^{2})\Big),
\]
the exponential factor being independent of $\theta$ up to the $O(r^{2})$ term.
Hence the left side of \eqref{eq:arcbound2} at radius $r_{R}$ is at most
$r_{R}M(r_{R})\exp(-c\,r_{R}^{2}\ln r_{R}+O(r_{R}^{2}))$, which tends to $0$
because $r_{R}\ge R$ and $\log^{+}M(r_{R})=o(R^{2}\ln R)$. Taking $R=3^{j}$ and
$R_{j}=r_{2^{j}}$ gives the sequence. The proofs of Theorems~\ref{thm:B} and
\ref{thm:D} use \eqref{eq:arcbound} only through the vanishing of the arc
integral, so \eqref{eq:arcbound2} may be substituted for it verbatim.
\end{proof}

\begin{remark}[On \Harc]\label{rem:Harc}
In the model case in which every $c$ is a rational combination of $1$ and
$\tau$, the logarithm of $\Pi$ reduces to a sum of the form
$\sum_{j\le M}\log|2\sin(\pi j\theta)|^{-1}$ with $\theta\in\{1/\tau,1\}$ ---
the classical small-denominator sum of the metric theory of continued fractions,
whose size is $O(M\ln M)$ for $\theta$ of bounded type --- and whose \emph{mean}
over a period vanishes, since $\int_{0}^{1}\ln|2\sin\pi t|\,dt=0$. The vanishing
mean is the reason one expects the $o(R^{2}\ln R)$ of \Harc to be very far from
sharp. We do \emph{not} claim the reduction in general, we make no attempt to
determine the optimal form of the bound, and we do not claim a proof of \Harc for
any specific $\tau$.

Three numerical facts, all for the running example \eqref{eq:example}, bear on
this. \emph{(a)} The split \eqref{eq:split} itself is confirmed: computing
$\ln|\Pi|$ from the double sine directly and from
$\ln|\Gs|-\Rea L\ln2\pi-\ln|\widetilde\Gs|$ agrees to
$7.6\times10^{-21}$ (check R6a). \emph{(b)} On the eastern arcs
$\max_{|\theta|\le\eps_{1}}\ln|\Pi(Re^{i\theta})|$ stays \emph{bounded} --- it
takes the values $1.41$, $0.52$, $-0.13$, $0.12$, $0.45$ at $R=4,6,8,10,12$ ---
so it grows neither like $R\ln R$ nor like $R^{2}$, and is consistent there with
exponent $0$ in \Harc (check R6). A nine-point angular sample cannot exclude a spike
between samples, so this is evidence, not proof. \emph{(c)} The average that
\textup{(H$^{\mathrm{int}}_{\mathrm{arc}}$)} bounds is small and falls
super-exponentially: the mean of the eastern-arc integral over $r\in[R,2R]$ is
$1.4\times10^{-4}$, $5.7\times10^{-16}$, $9.7\times10^{-40}$ at $R=3,5,7$
(check R5). Finally, the sup form itself was probed directly:
$\min_{R'\in[R,R+1]}\max_{|\arg s|\le\eps_{1}}|\phi(s)z^{-s}|$ falls from
$2.1\times10^{-6}$ at $R=3$ to $6.4\times10^{-97}$ at $R=9$ (check R4). A finite
set of radii cannot verify a hypothesis about $R\to\infty$, and we do not present
them as doing so.
\end{remark}

\begin{theorem}[Residue expansions]\label{thm:B}
Let $N<0$, assume \eqref{eq:generic}, and assume \Harc. Then, for
every $z$ in the domain of Theorem~\ref{thm:A},
\begin{equation}\label{eq:resexp}
\Kr(z)=-\lim_{j\to\infty}\!\!\sum_{\substack{s_{0}\in\Lp\\ |s_{0}|<R_{j}}}\!\!
\Res_{s=s_{0}}\phi(s)z^{-s},
\qquad
\Kl(z)=+\lim_{j\to\infty}\!\!\sum_{\substack{s_{0}\in\Lm\\ |s_{0}|<R_{j}}}\!\!
\Res_{s=s_{0}}\phi(s)z^{-s},
\end{equation}
the limits existing along the sequence $R_{j}$ of
\textup{(H$^{\sup}_{\mathrm{arc}}$)}, or along the sequence produced by
Lemma~\ref{lem:adaptive} if instead
\textup{(H$^{\mathrm{int}}_{\mathrm{arc}}$)} is assumed.

The summation order matters and is part of the statement. $\Lp$ and $\Lm$ are
two-dimensional lattices, so an unordered symbol $\sum_{s_{0}\in\Lp}$ has no
meaning unless the family is absolutely summable; what the proof delivers is the
\emph{exhaustion induced by the closing contours}, namely the limit of the
partial sums over $|s_{0}|<R_{j}$. If in addition the family is absolutely
summable --- which is what the Diophantine discussion of
Remark~\ref{rem:diophantine} is about --- the limit is independent of the
exhaustion and one may write the sums without qualification. Explicitly, for $s_0=a_j+l\tau+k\in\Lp$ with $\rho=l\tau+k$,
\begin{equation}\label{eq:resplus}
\Res_{s=s_0}\phi(s)z^{-s}
=\frac{-1}{G'(-\rho;\tau)}\;
\widehat{\Gs}_{j}(s_{0})\;e^{\pi\alpha s_{0}^{2}/\tau}z^{-s_{0}},
\end{equation}
where $\widehat{\Gs}_{j}$ denotes $\Gs$ with the factor $G(a_j-s;\tau)$ deleted
from the denominator, and $G'(-\rho;\tau)$ is given by \eqref{eq:Gprime}. For
$s_{0}=b_{j}-l\tau-k\in\Lm$ with $\rho'=(l-1)\tau+(k-1)$ the same formula holds
with $-1/G'(-\rho;\tau)$ replaced by $+1/G'(-\rho';\tau)$ and
$\widehat{\Gs}_{j}$ the analogous deletion of $G(1+\tau-b_j+s;\tau)$.
\end{theorem}

\begin{proof}
By Definition~\ref{def:four} the contour of $\Kr$ leaves the domain containing
large positive reals --- which by Proposition~\ref{prop:combs} contains all of
$\Lp$ --- on its \emph{right}, so it encircles $\Lp$ in the negative sense.
Realise $\gamma$ in the normal form of Lemma~\ref{lem:isotopy}: two rays from
$\sigma_{0}$ at angles $\pm\eps_{0}$. Such a ray meets the circle $|s|=R$ at a
point whose \emph{polar} angle is $\pm\eps_{0}+O(|\sigma_{0}|/R)$, so by
\eqref{eq:epsconv} there is $R_{0}$ with the property that for $R\ge R_{0}$ the
shorter arc of $|s|=R$ joining the two intersection points lies entirely in
$\{|\arg s|\le\eps_{1}\}$ --- this is exactly what the margin
$\eps_{1}>\eps_{0}$ buys. Close $\gamma$ at radius $R_{j}\ge R_{0}$ by that arc,
to which Lemma~\ref{lem:arcs} applies, and apply Cauchy's theorem to the
resulting bounded region, which contains exactly the points of $\Lp$ with
$|s_0|<R_{j}$. Letting
$j\to\infty$, the arc contribution vanishes by \eqref{eq:arcbound} and the ray
contributions converge absolutely by Theorem~\ref{thm:A}; the first identity of
\eqref{eq:resexp} follows, together with the convergence of the partial sums
over $|s_{0}|<R_{j}$. The second identity is the same argument on the western
arcs.

For \eqref{eq:resplus}, put $u=s-s_0$; then $a_j-s=-\rho-u$, so by
Lemma~\ref{lem:Gprime}
$G(a_j-s;\tau)=-G'(-\rho;\tau)\,u+O(u^{2})$, whence
$1/G(a_j-s;\tau)=-1/\big(G'(-\rho;\tau)u\big)+O(1)$. For $\Lm$ the vanishing
factor is $G(1+\tau-b_j+s;\tau)$, whose argument has derivative $+1$ in $s$;
this is the origin of the sign difference between the two cases.
\end{proof}

\begin{remark}[Diophantine caveat]\label{rem:diophantine}
The terms of \eqref{eq:resexp} carry a super-exponential factor
$\exp\big(-\tfrac{|N|}{2\tau}\rho^{2}\ln\rho+\tfrac{\pi\Rea\alpha}{\tau}\rho^{2}
+O(\rho)\big)$, but also products of reciprocal sines arising from
$D_{\tau}(l,k)$ in \eqref{eq:Gprime}: these are the same small denominators that
make the series representation of the density of the supremum of a stable process
converge only for $\tau$ outside a dense, Lebesgue-null, Hausdorff-dimension-zero
set of Liouville type \cite{HubalekKuznetsov2011,Kuznetsov2013,%
HackmannKuznetsov2013}. The super-exponential factor here is much stronger than
the $e^{-c\rho}$ available in the balanced case of \cite{KK2024}, so we expect
absolute summability for a much larger set of $\tau$; but we have verified it
only for $\tau=(\sqrt5-1)/2$, and a sharp Diophantine criterion remains open.

We deliberately do \emph{not} add the usual escape clause that, absent
convergence, \eqref{eq:resexp} is still a complete asymptotic expansion. That
statement requires a remainder estimate --- a bound on the contour integral over
the truncated part --- which we have not proved, and which is not a formal
consequence of anything above.
\end{remark}

\section{The linear relation and the collapse for real data}

Fix $\sigma_{0}$ with $\lambda_{\sharp}<\sigma_{0}<\lambda_{\flat}$ and any
$\eps\in(0,\pi/4)$ --- no closing arc occurs here, so the convention
\eqref{eq:epsconv} is not needed --- and define the four rays emanating from
$\sigma_{0}$
\[
R^{E}_{\pm}=\int_{\sigma_{0}}^{\;\sigma_{0}+\infty e^{\pm i\eps}}\phi(s)z^{-s}ds,
\qquad
R^{W}_{\pm}=\int_{\sigma_{0}}^{\;\sigma_{0}+\infty e^{\pm i(\pi-\eps)}}\phi(s)z^{-s}ds .
\]
Each of the four lies entirely in a decay sector, so each converges absolutely by
Theorem~\ref{thm:A}; \emph{no estimate outside $S^{\pm}_{\eps}$ is used}, and in
particular Lemma~\ref{lem:arcs} and hypothesis \Harc are not
needed here. Because $\sigma_{0}$ separates the two combs, the concatenations
below are admissible in the sense of Definition~\ref{def:admissible} and realise
the four classes of Definition~\ref{def:four}:
\begin{equation}\label{eq:fourpieces}
\begin{aligned}
2\pi i\,\Kr&=-R^{E}_{-}+R^{E}_{+}, &\qquad
2\pi i\,\Kl&=-R^{W}_{-}+R^{W}_{+},\\
2\pi i\,\Ku&=-R^{E}_{-}+R^{W}_{+}, &\qquad
2\pi i\,\Kd&=-R^{W}_{-}+R^{E}_{+}.
\end{aligned}
\end{equation}

\begin{theorem}[The linear relation]\label{thm:C}
Let $N<0$. Then, identically in $z$ and in all parameters,
\begin{equation}\label{eq:relation}
\Ku+\Kd=\Kr+\Kl .
\end{equation}
The identity is \emph{unconditional}: it uses only Theorem~\ref{thm:A}.
Moreover, setting
$\widehat A=\tfrac{1}{2\pi i}\big(R^{W}_{+}-R^{E}_{+}\big)$ and
$\widehat B=\tfrac{1}{2\pi i}\big(R^{W}_{-}-R^{E}_{-}\big)$,
\begin{equation}\label{eq:AB}
\Ku=\Kr+\widehat A=\Kl+\widehat B,\qquad
\Kd=\Kl-\widehat A=\Kr-\widehat B,\qquad
\widehat A-\widehat B=\Kl-\Kr .
\end{equation}
All of \eqref{eq:relation} and \eqref{eq:AB} is unconditional. If in addition
\Harc holds and the exhaustions of Theorem~\ref{thm:B} converge,
then and only then may one also write
\begin{equation}\label{eq:ABres}
\widehat A-\widehat B=\Kl-\Kr
=\lim_{j\to\infty}\!\!\sum_{\substack{s_{0}\in\Lp\cup\Lm\\ |s_{0}|<R_{j}}}\!\!
\Res_{s=s_{0}}\phi(s)z^{-s}.
\end{equation}
\end{theorem}

\begin{proof}
Adding the two expressions on the second line of \eqref{eq:fourpieces} and
comparing with the sum of the two on the first line gives \eqref{eq:relation};
subtracting them pairwise gives \eqref{eq:AB}; all of this is algebra on four
absolutely convergent integrals. Formula \eqref{eq:ABres} is a different
statement: it identifies $\widehat A-\widehat B$ with a residue sum, which is
Theorem~\ref{thm:B}, and is therefore conditional on \Harc.
\end{proof}

\begin{remark}[Dimension: what is and is not proved]\label{rem:three}
Relation \eqref{eq:relation} gives
\[
\dim\operatorname{span}\big\{\Kr,\Kl,\Ku,\Kd\big\}\le3 .
\]
We do \emph{not} prove equality: that would require the linear independence of
three of the four, which pairwise distinctness does not imply.

What can be said is this. The obstruction to a second relation is the
transcendent $\widehat A=\Ku-\Kr$, an integral along a ray pair hugging a decay
sector; it is not a residue sum, and we know no closed form for it. Numerically,
for the parameters of \S\ref{sec:numerics} the evaluation determinant
$\det\big[K_{i}(z_{r})\big]_{i,r=1}^{3}$ with
$(K_1,K_2,K_3)=(\Kr,\Kl,\Ku)$ and $z\in\{0.5,1.6,4.0\}$ has modulus
$2.4\times10^{-3}$ relative to the cube of the largest entry (check R3 of
Table~\ref{tab:R}), so no linear relation among the three can hold at those
three points simultaneously. That is evidence, not a proof: a proof would need
an analytic non-vanishing argument, for instance three asymptotic regimes in
which the three functions have distinct leading behaviour and the corresponding
coefficient matrix is invertible. We leave it open
(\S\ref{sec:open}, item~\ref{op:dim}).
\end{remark}

\begin{corollary}[Reality collapse]\label{cor:real}
Suppose $a_j,b_j,\tau,\alpha$ are all real and $z>0$. Then
$\overline{\phi(\bar s)}=\phi(s)$, and consequently
\begin{equation}\label{eq:realcollapse}
\Kr,\ \Kl\in\R,\qquad \Kd=\overline{\Ku},\qquad
\Rea\Ku=\tfrac12\big(\Kr+\Kl\big).
\end{equation}
Hence \emph{at most} three real degrees of freedom survive --- $\Kr$, $\Kl$ and
$\Ima\Ku$ --- the real part of $\Ku$ being forced by \eqref{eq:relation}.
Whether all three are independent is the question left open in
Remark~\ref{rem:three}.
\end{corollary}

\begin{proof}
All parameters being real, $\Gs$ and $e^{\pi\alpha s^{2}/\tau}$ are real on
$\R$, hence $\overline{\phi(\bar s)}=\phi(s)$ by the Schwarz reflection
principle; and $z^{-\bar s}=\overline{z^{-s}}$ for $z>0$. Reflection in the real
axis maps the contour of $\Kr$ (resp.\ $\Kl$) to itself with reversed
orientation and maps that of $\Ku$ to that of $\Kd$. The first two statements
follow; the third is \eqref{eq:relation}.
\end{proof}

\section{Vanishing theorems}

\begin{theorem}[Vanishing]\label{thm:D}
Let $N<0$ and assume \Harc.
\begin{enumerate}[label=\textup{(\roman*)},leftmargin=2.2em]
\item If $n=p$, i.e.\ $\Lp=\varnothing$, then $\Kr\equiv0$ identically on the
whole domain, and $\Ku+\Kd=\Kl$.
\item If $m=q$, i.e.\ $\Lm=\varnothing$, then $\Kl\equiv0$ identically, and
$\Ku+\Kd=\Kr$.
\end{enumerate}
\end{theorem}

\begin{proof}
In case (i) the region bounded by $\gamma$ and the eastern arc of radius
$R_{j}$ contains no pole, so Cauchy's theorem equates $\int_{\gamma}$ with the
arc integral, which tends to $0$ by \eqref{eq:arcbound}. Case (ii) is symmetric.
The two consequences are \eqref{eq:relation}, which is unconditional.
\end{proof}

\begin{remark}
Theorem~\ref{thm:D} is strictly stronger than its $N=0$ counterpart
\cite[Thm.~3]{KK2024}, which asserts $K^{0,p}_{p,p}=0$ only for $z>1$ (and
$K^{p,0}_{p,p}=0$ only for $0<z<1$) and is proved by a
Phragm\'en--Lindel\"of argument in the logarithmic case. Here the decay is
super-exponential and the conclusion is global. Note also that $n=p$ forces
$N=p+2m-q$, so $N<0$ requires $q>p+2m$: the vanishing case is non-empty.
\end{remark}

\section{Transformation laws}

\begin{theorem}[Transformations]\label{thm:E}
Let $N<0$ and let $\bullet$ denote any one of $\rightarrow,\leftarrow,\uparrow,
\downarrow$.
\begin{enumerate}[label=\textup{(E\arabic*)},leftmargin=2.6em,itemsep=4pt]
\item \textup{(Shift.)} For every $c\in\C$,
\begin{equation}\label{eq:E1}
K^{\bullet}\!\left(\begin{matrix}c+\va\\ c+\vb\end{matrix}\middle|z;\tau,\alpha\right)
=e^{\pi\alpha c^{2}/\tau}z^{-c}\,
K^{\bullet}\!\left(\begin{matrix}\va\\ \vb\end{matrix}\middle|
e^{-2\pi\alpha c/\tau}z;\tau,\alpha\right).
\end{equation}
\item \textup{(Inversion; interchanges east and west.)}
\begin{equation}\label{eq:E2}
\Kr\!\left(\begin{matrix}\va\\ \vb\end{matrix}\middle|z\right)_{(m,n,p,q)}
=\Kl\!\left(\begin{matrix}1+\tau-\vb\\ 1+\tau-\va\end{matrix}
\middle|z^{-1}\right)_{(n,m,q,p)},
\end{equation}
together with the same identity with $\rightarrow$ and $\leftarrow$
interchanged, while $\Ku$ and $\Kd$ are each mapped to themselves.
\item \textup{(Modular.)} With $A,B$ as in \cite[Eq.~(17)]{KK2024},
\begin{equation}\label{eq:E3}
K^{\bullet}\!\left(\begin{matrix}\va\\ \vb\end{matrix}\middle|z;\tau,\alpha\right)
=A\,\tau\;
K^{\bullet}\!\left(\begin{matrix}\tau^{-1}\va\\ \tau^{-1}\vb\end{matrix}\middle|
(Bz)^{\tau};\tau^{-1},\alpha-\tfrac{N}{2\pi}\ln\tau\right).
\end{equation}
In particular $N$ is invariant under $\tau\mapsto\tau^{-1}$, so the class
$N<0$ is closed under \eqref{eq:E3}.
\end{enumerate}
\end{theorem}

\begin{proof}
(E1) The substitution $s\mapsto s+c$ replaces $\gamma$ by $\gamma+c$, a
translation, which changes neither the terminal ray angles nor the homotopy
class; the identity then follows from \cite[Eq.~(15)]{KK2024} exactly as in
\cite[Thm.~2]{KK2024}.

(E2) The substitution $s\mapsto-s$ sends a terminal ray of angle $\theta$ to one
of angle $\theta+\pi$ and reverses the orientation of the contour, so it acts on
the four unordered pairs of terminal angles by
$\{-\eps,\eps\}\leftrightarrow\{-\pi+\eps,\pi-\eps\}$ while fixing
$\{-\eps,\pi-\eps\}$ and $\{-\pi+\eps,\eps\}$. Combined with the first identity
of \cite[Eq.~(15)]{KK2024} and $z\mapsto z^{-1}$, and noting that
$e^{\pi\alpha s^{2}/\tau}$ is even, this gives the stated pairing.

(E3) The substitution $s=v\tau$ with $\tau>0$ is a positive real dilation, hence
preserves all ray angles and the homotopy class. Inserting
\cite[Eq.~(16)]{KK2024}, $\Gs(s)=A\,B^{-s}\tau^{-Ns^{2}/(2\tau)}
\Gs^{(\tau^{-1})}(s\tau^{-1})$, into \eqref{eq:Kdef} and using $ds=\tau\,dv$,
\[
K=\frac{A\tau}{2\pi i}\int \Gs^{(\tau^{-1})}(v)
\exp\!\Big(\pi\alpha\tau v^{2}-\tfrac{N\tau}{2}v^{2}\ln\tau\Big)
\big((Bz)^{\tau}\big)^{-v}dv,
\]
and matching $\pi\tilde\alpha v^{2}/\tilde\tau$ with $\tilde\tau=\tau^{-1}$
gives $\tilde\alpha=\alpha-\tfrac{N}{2\pi}\ln\tau$.
\end{proof}

\begin{remark}
The prefactor in \eqref{eq:E3} is the \emph{product} of the constant $A$ of
\cite[Eq.~(17)]{KK2024} with $\tau$, not $A$ raised to the power $\tau$. It is
typeset unambiguously in \cite[Eq.~(30)]{KK2024}; we flag it only because the
plain-text layer of some \textsc{pdf} extractions renders ``$A\tau$'' as a
superscript. Verified on all four contour classes for $\alpha\in\{0,0.25\}$ to
$5.4\times10^{-21}$ (T5c).
\end{remark}

\begin{remark}[Why $\alpha$ is there]\label{rem:whyalpha}
Identity \eqref{eq:E3} is also the reason the Gaussian factor
$e^{\pi\alpha s^{2}/\tau}$ belongs in \eqref{eq:phi} at all: unless $N=0$ it
shifts $\alpha$ by $-\tfrac{N}{2\pi}\ln\tau$, so a definition without $\alpha$
would not be closed under the modular law. This is observed by Karp and
Kuznetsov in the published version \cite[p.~238]{KK2024} and, for the
Fox--Barnes $J$-function, by Vaz \cite[Rem.~2]{Vaz2025}. In our setting it is
the first sign of the invariance recorded in
Theorem~\ref{thm:Ninvariant}: $N$ is unchanged by \eqref{eq:E3}, and it is
exactly the failure of $N$ to vanish that forces the extra parameter.
\end{remark}

\section{Failure of the vertical Mellin--Barnes representation}

\begin{theorem}[No admissible vertical contour]\label{thm:F1}
Let $N<0$ and let $\alpha\in\C$ be arbitrary. Then along either vertical
direction $\arg s=\pm\pi/2$,
\[
\big|\phi(s)z^{-s}\big|=\exp\!\Big(\tfrac{|N|}{2\tau}\,r^{2}\ln r\,(1+o(1))\Big)
\longrightarrow\infty ,
\]
so no contour $L_{\theta^-,\theta^+}$ with $\theta^{+}$ or $-\theta^{-}$ in
$(\pi/4,3\pi/4)$ satisfies \textup{(A4)} of Definition~\ref{def:admissible}. In
particular no straight line $c+i\R$ is admissible, and there is no strip in which
$\phi$ is recovered from a function by the classical Mellin inversion formula
along a vertical line.
\end{theorem}

\begin{proof}
Proposition~\ref{prop:sectors} with $\cos2\theta=-1$ and $N<0$.
\end{proof}

\begin{corollary}\label{cor:noThm67}
The proofs of \cite[Thms.~6 and 7]{KK2024} do not transfer to $N<0$: the first
computes $\int_{0}^{\infty}K(x)x^{s-1}dx$ by deforming $\gamma$ to $L_{i\infty}$,
and the second applies the Mellin convolution theorem to the resulting
transform. Both steps are unavailable by Theorem~\ref{thm:F1}.
\end{corollary}

\begin{remark}[What is \emph{not} claimed]\label{rem:notclaimed}
We do not assert that $K^{\bullet}$ fails to possess a Mellin transform. Two
things stand in the way of such a claim, and both are instructive.
\begin{enumerate}[label=\textup{(\roman*)},leftmargin=2.2em,itemsep=2pt]
\item \emph{Degenerate parameters.} Condition \eqref{eq:generic} does not exclude
$n=p$, and then $\Kr\equiv0$ by Theorem~\ref{thm:D}; its Mellin transform exists
trivially. Any negative statement must first exclude the vanishing cases.
\item \emph{Termwise divergence is not divergence.} One is tempted to argue from
Theorem~\ref{thm:B} that, since the exponents in the expansion of $\Kr$ are
unbounded below, $\Kr$ must grow faster than any power as $z\to0^{+}$. That does
not follow: the expansion is a convergent sum only under
\Harc and in a regime where it is an expansion \emph{at infinity},
term-by-term growth in the opposite limit says nothing about the sum, and
coincident exponents may produce cancellation in the leading coefficient. A
genuine proof would require two-sided asymptotics for $K^{\bullet}$ together with
the non-vanishing of the relevant leading coefficients. We leave that open
(\S\ref{sec:open}, item~\ref{op:mellin}).
\end{enumerate}
Theorem~\ref{thm:F1} is the negative statement we can prove, and it is already
enough to explain why the four non-vertical classes are unavoidable.
\end{remark}

The difference relations that drive \cite[Thm.~7]{KK2024} do survive, however ---
not as integral equations in $x$, but as recurrences on the residue coefficients,
which need no Mellin transform at all.
Recall from the proof of \cite[Thm.~7]{KK2024} that
$\Gs(s+1)=F(s)\Gs(s)$ and $\Gs(s+\tau)=H(s)\Gs(s)$ with
\begin{align}
F(s)&=\frac{\prod_{j=1}^{n}\Gamma\!\big(1+\tfrac{1-a_j+s}{\tau}\big)
            \prod_{j=n+1}^{p}\Gamma\!\big(\tfrac{a_j-s-1}{\tau}\big)}
           {\prod_{j=1}^{m}\Gamma\!\big(\tfrac{b_j-s-1}{\tau}\big)
            \prod_{j=m+1}^{q}\Gamma\!\big(1+\tfrac{1-b_j+s}{\tau}\big)},
\label{eq:F}\\[4pt]
H(s)&=(2\pi)^{\frac{(\tau-1)(p-q)}{2}}\tau^{\mu-N(\tau+\frac12)}\tau^{-sN}
\frac{\prod_{j=1}^{n}\Gamma(1+\tau-a_j+s)\prod_{j=n+1}^{p}\Gamma(a_j-\tau-s)}
     {\prod_{j=1}^{m}\Gamma(b_j-\tau-s)\prod_{j=m+1}^{q}\Gamma(1+\tau-b_j+s)}.
\label{eq:H}
\end{align}

\begin{theorem}[Coefficient recurrence]\label{thm:F}
Assume \eqref{eq:generic} and set, for $n<j\le p$ and $l,k\ge0$,
\[
c^{(j)}_{l,k}:=\Res_{s=a_j+l\tau+k}\Gs(s),
\qquad s_{0}=a_j+l\tau+k .
\]
If $F$ and $H$ are regular at $s_{0}$ --- which holds for generic parameters ---
then
\begin{equation}\label{eq:recur}
c^{(j)}_{l,k+1}=F(s_{0})\,c^{(j)}_{l,k},
\qquad
c^{(j)}_{l+1,k}=H(s_{0})\,c^{(j)}_{l,k}.
\end{equation}
The analogous coefficients on $\Lm$ obey the inverse relations
$c^{-}_{l,k+1}=c^{-}_{l,k}/F(s_{0}-1)$ and
$c^{-}_{l+1,k}=c^{-}_{l,k}/H(s_{0}-\tau)$. The two recurrences are compatible:
\begin{equation}\label{eq:compat}
F(s+\tau)\,H(s)=H(s+1)\,F(s),
\end{equation}
so the value of $c^{(j)}_{l,k}$ does not depend on the lattice path used to
reach it --- \emph{path-independence holds on every connected subset of the
lattice $\{(l,k):l,k\ge0\}$ on which all the multipliers appearing in
\eqref{eq:recur} are finite and non-zero}, which is where the recurrence is
defined in the first place. We call such a subset a \emph{regular region}.
\end{theorem}

\begin{proof}
$\Res_{s=s_0}\Gs(s+1)=\Res_{s=s_0+1}\Gs(s)=c^{(j)}_{l,k+1}$, while
$\Res_{s=s_0}\big(F(s)\Gs(s)\big)=F(s_0)\,c^{(j)}_{l,k}$ whenever $F$ is regular
at $s_0$. The second relation is identical with $1$ replaced by $\tau$. On
$\Lm$ the lattice is generated by $-1$ and $-\tau$, whence the inverses. For
\eqref{eq:compat}, both sides equal $\Gs(s+1+\tau)/\Gs(s)$, computed along the
two orders of the two shifts; the identity is therefore automatic wherever all
four factors are finite and non-zero, which is the regularity requirement in the
statement (\eqref{eq:generic} rules out the coincidences that would make a
$\Gamma$-factor of \eqref{eq:F}--\eqref{eq:H} singular at a lattice point, but
it does not by itself force non-vanishing, so regularity is assumed separately).
Path-independence on a regular region follows by induction on any lattice path
inside it, since \eqref{eq:compat} lets one interchange two consecutive steps.
\end{proof}

\begin{remark}[Why this is the right replacement]
The route of \cite[Thm.~7]{KK2024} is to convert
$\Gs(s+1)=F(s)\Gs(s)$ into an integral equation in $x$ using the Mellin
convolution theorem; that step uses a Mellin transform represented and inverted
on a vertical line, which Theorem~\ref{thm:F1} shows is unavailable here. The same difference relation
nevertheless holds pointwise on residues. The practical consequence is
substantial: only the $p-n$ seeds $c^{(j)}_{0,0}$ require an evaluation of
$G'$ and of $\Gs$; every other coefficient of the two-dimensional array is
generated by \eqref{eq:recur} using ordinary gamma functions only. Evaluating a
truncated double series to order $L$ therefore costs $O(L^{2})$ gamma
evaluations rather than $O(L^{2})$ double-gamma evaluations. The reduction to
$p-n$ seeds presumes the simple situation in which $\tau\notin\Q$, no two poles
of $\Lp$ coincide, no numerator zero cancels one of them, and $F,H$ are finite
and non-zero at every lattice point traversed; outside it the array splits into
several orbits, or some multipliers degenerate, and more seeds are needed.
\end{remark}

\section{\texorpdfstring{Kernel calculus and partial closure under the classical transforms}{Kernel calculus and partial closure under the classical transforms}}\label{sec:closure}

The usefulness of the Meijer $G$-function rests on a list of \emph{stability
properties}: the family is preserved by integer and fractional differentiation
and integration, by multiplication by powers, by inversion of the argument, by
the Laplace and Euler transforms, and by Mellin convolution
\cite{MathaiSaxena1973,Kiryakova2021}. Karp and Kuznetsov open \cite{KK2024} by
recalling this list, but do not ask which parts of it survive for the
$K$-function. This section settles that question. The answer is uniform, and it
produces a structural invariant which explains a puzzling feature of the whole
subject (Corollary~\ref{cor:Ntrap}).

\subsection{The mechanism}
Every operation on the list acts on the Mellin transform by multiplication by a
\emph{ratio of ordinary gamma functions}:
\begin{equation}\label{eq:multipliers}
\begin{aligned}
\text{Erd\'elyi--Kober } I^{\gamma,\delta}_{\beta=1}&:\
\frac{\Gamma(\gamma+1-s)}{\Gamma(\gamma+\delta+1-s)},
&\qquad
\text{Euler}&:\ \Gamma(\gamma)\frac{\Gamma(\beta-s)}{\Gamma(\beta+\gamma-s)},
\\[2pt]
\text{Laplace}&:\ \Gamma(s),\ \text{with }s\mapsto1-s,
&\qquad
\tfrac{d}{dz}&:\ -s=-\tfrac{\Gamma(s+1)}{\Gamma(s)},\ \text{with }s\mapsto s+1,
\end{aligned}
\end{equation}
while Mellin convolution multiplies by the other function's Mellin transform.
For the Meijer $G$-function the Mellin-side object is itself a ratio of gamma
functions, so the family is trivially preserved. For the $K$-function the
Mellin-side object $\phi(s)=\Gs(s)e^{\pi\alpha s^{2}/\tau}$ is a ratio of
\emph{double} gammas, and the question is whether multiplying it by a ratio of
\emph{ordinary} gammas stays inside the family. It does, because of
\eqref{eq:GammaviaG}: each gamma factor is a ratio of two $G$'s, i.e.\ exactly
one new $a$-parameter and one new $b$-parameter, plus a factor $\tau^{\pm s}$
which is a rescaling of $z$.

\begin{lemma}[Insertion rules]\label{lem:insert}
Let $\phi$ be as in \eqref{eq:phi} and put $\kappa=(2\pi)^{(1-\tau)/2}$. For
every $c\in\C$ and each multiplier $R$ in the table below,
\begin{equation}\label{eq:insert}
R(s)\,\phi^{m,n}_{p,q}\!\left(\begin{matrix}\va\\ \vb\end{matrix}\middle|s\right)
= C\,\lambda^{-s}\;
\phi^{m',n'}_{p',q'}\!\left(\begin{matrix}\va'\\ \vb'\end{matrix}\middle|s\right),
\qquad\text{equivalently}\qquad
\frac{1}{2\pi i}\int_{\gamma}\! R\,\phi\,z^{-s}ds = C\,K'(\lambda z),
\end{equation}
with $\alpha$ unchanged and the data given by
\smallskip

\centerline{\begin{tabular}{@{}lcccccc@{}}
\toprule
& & & \multicolumn{2}{c}{new $a$} & \multicolumn{2}{c}{new $b$}\\
\cmidrule(lr){4-5}\cmidrule(lr){6-7}
$R(s)$ & $C$ & $\lambda$ & value & block & value & block\\
\midrule
$\Gamma(c-s)$   & $\kappa\,\tau^{c-\frac12}$    & $\tau$
  & $c$        & den. & $c+\tau$   & num.\\
$\Gamma(c+s)$   & $\kappa\,\tau^{c-\frac12}$    & $\tau^{-1}$
  & $1-c$      & num. & $1+\tau-c$ & den.\\
$1/\Gamma(c-s)$ & $\kappa^{-1}\tau^{\frac12-c}$ & $\tau^{-1}$
  & $c+\tau$   & den. & $c$        & num.\\
$1/\Gamma(c+s)$ & $\kappa^{-1}\tau^{\frac12-c}$ & $\tau$
  & $1+\tau-c$ & num. & $1-c$      & den.\\
\bottomrule
\end{tabular}}
\smallskip

\noindent
Here a new $a$-parameter in the numerator block increments $n$ and $p$; in the
denominator block it increments $p$ only. Likewise a new $b$-parameter in the
numerator block increments $m$ and $q$; in the denominator block it increments
$q$ only.
\end{lemma}

\begin{proof}
Apply \eqref{eq:GammaviaG} with $s$ replaced by $c-s$:
\[
\Gamma(c-s)=\kappa\,\tau^{c-s-\frac12}\,
\frac{G(c+\tau-s;\tau)}{G(c-s;\tau)} .
\]
Comparing with \eqref{eq:Gscript}, the numerator $G(c+\tau-s;\tau)$ is of the
type $G(b_j-s;\tau)$ occurring for $j\le m$, so it is a new $b$-parameter
$c+\tau$ in the numerator block; the denominator $G(c-s;\tau)$ is of the type
$G(a_j-s;\tau)$ occurring for $j>n$, so it is a new $a$-parameter $c$ in the
denominator block. The remaining factor $\kappa\tau^{c-\frac12}\tau^{-s}$ gives
$C$ and $\lambda=\tau$, since multiplying the integrand by $\tau^{-s}$ replaces
$z$ by $\tau z$. The other three rows are identical computations using
$\Gamma(c+s)=\kappa\tau^{c+s-\frac12}G(c+\tau+s;\tau)/G(c+s;\tau)$ together with
$G(c+\tau+s;\tau)=G(1+\tau-(1-c)+s;\tau)$ and
$G(c+s;\tau)=G(1+\tau-(1+\tau-c)+s;\tau)$, and by taking reciprocals.
Since no multiplier contains a factor $e^{\text{const}\cdot s^{2}}$, $\alpha$ is
untouched.
\end{proof}

\subsection{The transforms}

We separate two statements that are easy to conflate: an algebraic one about
the \emph{kernel}, which is unconditional, and an analytic one about the
\emph{integral transform}, which is not.

\begin{theorem}[Kernel calculus]\label{thm:kernelcalc}
Let $R$ be any finite product of factors $\Gamma(c\mp s)^{\pm1}$. Then
$R(s)\phi(s)=C\lambda^{-s}\phi_{\mathrm{new}}(s)$ with $C,\lambda$ and the
enlarged parameter lists obtained by composing the rows of
Lemma~\ref{lem:insert}; the multipliers \eqref{eq:multipliers} of the
Erd\'elyi--Kober, Euler, Laplace and differentiation operators are all of this
form. Moreover the product $\phi_{1}\phi_{2}$ of two such kernels with a common
$\tau$ is again of the form \eqref{eq:phi}, with the parameter blocks
concatenated and $\alpha=\alpha_{1}+\alpha_{2}$.
\end{theorem}

\begin{proof}
Immediate from Lemma~\ref{lem:insert} and, for the product, from the fact that
$\Gs_1\Gs_2$ is a $\Gs$ and
$e^{\pi\alpha_1s^2/\tau}e^{\pi\alpha_2s^2/\tau}=e^{\pi(\alpha_1+\alpha_2)s^2/\tau}$.
\end{proof}

Theorem~\ref{thm:kernelcalc} is a statement about integrands. Turning it into an
identity between functions of $x$ requires, for each operator and each contour
class, that the elementary integral defining the multiplier converge
\emph{along the contour} and that Fubini's theorem apply. That is a genuine
restriction, and it is not uniform over the four classes.

\begin{theorem}[Analytic transform identities]\label{thm:closure}
Let $N<0$ and let $\bullet$ be one of the four classes of
Definition~\ref{def:four}, realised as in \eqref{eq:fourpieces} with abscissa
$\sigma_{0}$. Assume throughout that the enlarged parameter set still satisfies
$\lambda_{\sharp}<\lambda_{\flat}$. Then:
\begin{enumerate}[label=\textup{(\roman*)},leftmargin=2.4em,itemsep=6pt]

\item \textup{(Euler / beta transform; $\bullet=\leftarrow$ only.)} The inner
integral $\int_{0}^{1}t^{\beta-s-1}(1-t)^{\gamma-1}dt$ converges iff
\emph{both} $\Rea(\beta-s)>0$ (endpoint $t=0$) and $\Rea\gamma>0$
(endpoint $t=1$). The second is a hypothesis on the operator and we impose it.
For the first: on any class with an eastern end $\Rea s\to+\infty$, so it
fails; on $\gamma_{\leftarrow}$ one has $\Rea s\le\sigma_{0}$
throughout, and it holds as soon as $\Rea\beta>\sigma_{0}$. Assume therefore
\begin{equation}\label{eq:eulerhyp}
\Rea\beta>\sigma_{0},\qquad \Rea\gamma>0 .
\end{equation}
Then
\begin{equation}\label{eq:euler}
\begin{aligned}
\int_{0}^{1}&t^{\beta-1}(1-t)^{\gamma-1}\,
K^{\bullet}\!\left(\begin{matrix}\va\\ \vb\end{matrix}\middle|zt;\tau,\alpha\right)dt\\
&\qquad=\Gamma(\gamma)\,\tau^{-\gamma}\,
K^{\bullet}{}^{m+2,\,n}_{\,p+2,q+2}\!
\left(\begin{matrix}\va,\ \beta,\ \beta+\gamma+\tau\\
                    \beta+\tau,\ \beta+\gamma,\ \vb\end{matrix}
\middle|z;\tau,\alpha\right),
\end{aligned}
\end{equation}
with $\bullet=\leftarrow$ on both sides, the new $b$'s numerator-type and the
new $a$'s denominator-type. By the inversion law \textup{(E2)} the mirror
identity --- an integral of $\Kr(z/t)$ over $t\in(0,1)$ --- holds for $\Kr$ with
the reflected parameters.

\item \textup{(Erd\'elyi--Kober, hence Riemann--Liouville;
$\bullet=\leftarrow$.)} This is \textup{(i)} with $\beta=\gamma+1$ and
$\gamma$ there replaced by $\delta$, so \eqref{eq:eulerhyp} becomes
\begin{equation}\label{eq:ekhyp}
\Rea(\gamma+1)>\sigma_{0},\qquad \Rea\delta>0 .
\end{equation}
With
$I^{\gamma,\delta}_{1}f(z)=\frac{1}{\Gamma(\delta)}\int_{0}^{1}
(1-\sigma)^{\delta-1}\sigma^{\gamma}f(z\sigma)\,d\sigma$,
\begin{equation}\label{eq:ek}
I^{\gamma,\delta}_{1}K^{\bullet}(z)=\tau^{-\delta}\,
K^{\bullet}{}^{m+2,\,n}_{\,p+2,q+2}\!
\left(\begin{matrix}\va,\ \gamma+1,\ \gamma+\delta+1+\tau\\
                    \gamma+1+\tau,\ \gamma+\delta+1,\ \vb\end{matrix}
\middle|z\right).
\end{equation}
The Riemann--Liouville integral is $I^{\delta}_{0+}=z^{\delta}I^{0,\delta}_{1}$,
so it is closed by \eqref{eq:ek} together with \textup{(E1)}. Fractional
\emph{derivatives} are treated separately, and without any continuation in
$\delta$, in Theorem~\ref{thm:RL}.

\item \textup{(Differentiation of every order; all four classes.)} The
multiplier of $d^{k}/dz^{k}$ is $(-1)^{k}\Gamma(s+k)/\Gamma(s)$ together with the
factor $z^{-k}$, since
$d^{k}(z^{-s})/dz^{k}=(-1)^{k}\big(\Gamma(s+k)/\Gamma(s)\big)z^{-s-k}$. It is
entire and polynomially bounded, so differentiation under the integral sign is
justified for every $\bullet$ by the absolute convergence of
Theorem~\ref{thm:A}, and rows two and four of Lemma~\ref{lem:insert} with
$c=k$ and $c=0$ give, for every $k\in\Z_{\ge1}$,
\begin{equation}\label{eq:derivk}
\frac{d^{k}}{dz^{k}}K^{\bullet}\!\left(\begin{matrix}\va\\ \vb\end{matrix}\middle|z\right)
=\frac{(-\tau)^{k}}{z^{k}}\,
K^{\bullet}{}^{m,\,n+2}_{\,p+2,q+2}\!
\left(\begin{matrix}1+\tau,\ 1-k,\ \va\\
                     \vb,\ 1+\tau-k,\ 1\end{matrix}\middle|z\right),
\end{equation}
the new $a$'s numerator-type and the new $b$'s denominator-type. For $k=1$
this is
\begin{equation}\label{eq:deriv}
\frac{d}{dz}K^{\bullet}\!\left(\begin{matrix}\va\\ \vb\end{matrix}\middle|z\right)
=-\frac{\tau}{z}\,
K^{\bullet}{}^{m,\,n+2}_{\,p+2,q+2}\!
\left(\begin{matrix}0,\ 1+\tau,\ \va\\ \vb,\ \tau,\ 1\end{matrix}\middle|z\right).
\end{equation}
Formula \eqref{eq:derivk} is \emph{not} obtained by iterating \eqref{eq:deriv}:
iteration would differentiate the product $z^{-1}K$ and produce a sum of $k$
terms. Inserting the full multiplier once gives the single term.

\item \textup{(Laplace transform and Mellin convolution: kernel level only.)}
The Laplace multiplier is $\Gamma(s)$ together with the reflection
$s\mapsto1-s$, and the Mellin-convolution multiplier is the second kernel
$\phi_{2}$; by Theorem~\ref{thm:kernelcalc} both keep us inside the family,
giving in the Laplace case order $(m,n+1,p+1,q+1)$ relative to the reflected
data and in the convolution case
$(m_{1}+m_{2},n_{1}+n_{2},p_{1}+p_{2},q_{1}+q_{2})$ with
$\alpha=\alpha_{1}+\alpha_{2}$. We do \emph{not} establish the corresponding
identities between functions of $x$. For the Laplace transform this would need
$\int_{0}^{\infty}e^{-pt}t^{-s}dt$ to converge along $\gamma_{\bullet}$, i.e.\
$0<\Rea s<1$ there, which no class satisfies; for the convolution
\begin{equation}\label{eq:mconv}
\int_{0}^{\infty}K_{1}\!\left(\tfrac{x}{t}\right)K_{2}(t)\,\frac{dt}{t}
\end{equation}
it would need a common fundamental strip for the two $K$-functions; no such
strip has been established, because the two-sided asymptotics that would produce
one remain open (Remark~\ref{rem:notclaimed}). We therefore record these two as \emph{formal} identities, valid
at the level of Theorem~\ref{thm:kernelcalc}, and leave their analytic
counterparts open (\S\ref{sec:open}, item~\ref{op:transforms}).
\end{enumerate}
\end{theorem}

\begin{proof}
Each statement follows by inserting the corresponding multiplier of
\eqref{eq:multipliers} into the contour integral and applying
Lemma~\ref{lem:insert}. We give (i); (ii) is the same computation and (iii) is
differentiation under the integral sign. Interchanging the order of integration,
\begin{align*}
\int_{0}^{1}t^{\beta-1}(1-t)^{\gamma-1}K^{\bullet}(zt)\,dt
&=\frac{1}{2\pi i}\int_{\gamma}\phi(s)z^{-s}
\left[\int_{0}^{1}t^{\beta-s-1}(1-t)^{\gamma-1}dt\right]ds\\
&=\frac{1}{2\pi i}\int_{\gamma}
\frac{\Gamma(\beta-s)\Gamma(\gamma)}{\Gamma(\beta+\gamma-s)}\,\phi(s)z^{-s}ds,
\end{align*}
the inner beta integral converging under \eqref{eq:eulerhyp}, which holds
throughout $\gamma_{\leftarrow}$. The interchange itself needs more than the
absolute convergence of Theorem~\ref{thm:A}, because the beta kernel is an extra
factor; it is justified because that factor is only polynomially large. Indeed
$B(\beta-s,\gamma)=\Gamma(\beta-s)\Gamma(\gamma)/\Gamma(\beta+\gamma-s)
=O\big(|s|^{-\Rea\gamma}\big)$ uniformly as $s\to\infty$ along the two western
rays, by Stirling, whereas $|\phi(s)z^{-s}|$ decays like
$\exp(-c|s|^{2}\ln|s|)$ there (Theorem~\ref{thm:A}); the product is therefore
absolutely integrable on $\gamma_{\leftarrow}\times(0,1)$ and Fubini--Tonelli
applies. Applying the first and third rows of
Lemma~\ref{lem:insert} with $c=\beta$ and $c=\beta+\gamma$ gives the parameter
lists, the constant $\Gamma(\gamma)\cdot\kappa\tau^{\beta-\frac12}\cdot
\kappa^{-1}\tau^{\frac12-\beta-\gamma}=\Gamma(\gamma)\tau^{-\gamma}$, and the
scale $\lambda=\tau\cdot\tau^{-1}=1$; the two $\tau^{\pm s}$ factors cancel, so
the argument is unchanged. Item (iv) is Theorem~\ref{thm:kernelcalc}.
\end{proof}

\begin{remark}[Fubini replaces Mellin --- where it can]\label{rem:noMellin}
The usual derivation multiplies the Mellin transform and inverts; by
Theorem~\ref{thm:F1} that route is closed for $N<0$. The proof above substitutes
a different one: insert the multiplier \emph{under} the contour integral and
interchange, which absolute convergence (Theorem~\ref{thm:A}) permits. The price
is that the inner integral must converge \emph{on the contour}, and that is what
restricts (i)--(ii) to $\gamma_{\leftarrow}$ and blocks (iv) altogether. The
honest summary: the kernel calculus is unconditional and covers everything; the
analytic identities are proved for the operators and classes listed, and are
formal elsewhere.
\end{remark}

\begin{remark}[Two realisations of one multiplier, and which to prefer]
\label{rem:tworealisations}
Because $(-1)^{k}\Gamma(s+k)/\Gamma(s)=\Gamma(1-s)/\Gamma(1-s-k)$, the
derivative multiplier may equally be inserted through rows \emph{one and three}
of Lemma~\ref{lem:insert}, with $c=1$ and $c=1-k$. That is the route Vaz takes
for the Fox--Barnes $J$-function \cite[Prop.~5(i)]{Vaz2025}. It yields a
different parameter list --- of type $(m+2,n)$ rather than $(m,n+2)$ --- and a
$K$-function differing from the one in \eqref{eq:derivk} by the factor
$(-1)^{k}$, so that the two right-hand sides agree; verified to
$9.2\times10^{-23}$ (check CE1 of Table~\ref{tab:C}).

The two are nevertheless not interchangeable in practice. The second realisation
appends the $a$-parameter $1-k+\tau$ to the denominator block, which lowers
$\lambda_{\flat}$ to $\min(\lambda_{\flat},1-k+\tau)$ and therefore destroys
$\lambda_{\sharp}<\lambda_{\flat}$ once $k$ is large --- for the running example
\eqref{eq:example} already at $k=3$, where the enlarged parameters give
$\lambda_{\flat}=\tau-2<\lambda_{\sharp}$ and no admissible contour exists. The
first realisation appends nothing to the denominator $a$-block and therefore
leaves $\lambda_{\flat}$ untouched, but it adds the denominator $b$-parameters
$1+\tau-k$ and $1$. Hence it is admissible precisely when
\[
\max\{\lambda_{\sharp},-k,-\tau\}<\lambda_{\flat};
\]
it is usually more robust, but not automatically admissible for every $k$. A
simple sufficient condition valid for all $k\ge1$ is
$\max\{\lambda_{\sharp},-1,-\tau\}<\lambda_{\flat}$. When both
realisations are available, insert through $\Gamma(c+s)$.
\end{remark}

\subsection{Fractional derivatives}

For a real non-integral order $\delta>0$, put
$r=\lceil\delta\rceil$ and $\nu=r-\delta\in(0,1)$.  The
Riemann--Liouville derivative is the composition
\[
D^{\delta}_{0+}=\frac{d^{r}}{dz^{r}}\circ I^{\,\nu}_{0+}.
\]
Thus no analytic continuation in $\delta$ is needed: the integral leg is the
Erd\'elyi--Kober rule of Theorem~\ref{thm:closure}(ii), and the differential leg
is \eqref{eq:derivk}.  The shift law must, however, be used with its full
$\alpha$-dependent constant and argument scaling.

\begin{theorem}[Riemann--Liouville derivatives]\label{thm:RL}
Let $N<0$, let $\bullet=\leftarrow$ with abscissa $\sigma_{0}$, and let
$\delta\in\R_{>0}\setminus\Z$. Put
$r=\lceil\delta\rceil$ and $\nu=r-\delta\in(0,1)$. Assume
\begin{equation}\label{eq:RLhyp}
\lambda_{\sharp}<\sigma_{0}<\min(\lambda_{\flat},1),
\end{equation}
and that the Erd\'elyi--Kober, shifted, and final enlarged parameter sets each
satisfy the relevant separation condition
$\lambda_{\sharp}<\lambda_{\flat}$. Then, for $z>0$,
$I^{\,\nu}_{0+}K^{\leftarrow}$ exists and
\begin{equation}\label{eq:RL}
\begin{aligned}
D^{\delta}_{0+}K^{\leftarrow}\!\left(\begin{matrix}\va\\ \vb\end{matrix}\middle|z;\tau,\alpha\right)
&=\frac{(-1)^{r}\tau^{\,r-\nu}e^{\pi\alpha\nu^{2}/\tau}}{z^{r}}\\
&\quad\times
K^{\leftarrow}{}^{m+2,\,n+2}_{\,p+4,q+4}\!
\left(\begin{matrix}1+\tau,\ 1-r,\ \va-\nu,\ 1-\nu,\ 1+\tau\\
        1,\ 1+\tau-\nu,\ \vb-\nu,\ 1+\tau-r,\ 1\end{matrix}
\middle|e^{-2\pi\alpha\nu/\tau}z;\tau,\alpha\right),
\end{aligned}
\end{equation}
where in each list the first two displayed new entries are numerator-type and
the last two are denominator-type.  On the right-hand side the branch is fixed
by
$\log(e^{-2\pi\alpha\nu/\tau}z)=\log z-2\pi\alpha\nu/\tau$.
In particular $N$ is unchanged.
\end{theorem}

\begin{proof}
Since $\nu>0$ and $1>\sigma_{0}$ by \eqref{eq:RLhyp},
\eqref{eq:ekhyp} holds with $\gamma=0$ and $\delta=\nu$. Hence
Theorem~\ref{thm:closure}(ii) gives
$I^{0,\nu}_{1}K^{\leftarrow}=\tau^{-\nu}K^{\leftarrow}_{Q}$, where
$Q$ has order $(m+2,n,p+2,q+2)$ and parameter lists
$\va,1,1+\nu+\tau$ and $1+\nu,1+\tau,\vb$ (with the two new $a$'s
denominator-type and the two new $b$'s numerator-type). The substitution
$t=z\sigma$ gives
\[
I^{\,\nu}_{0+}K^{\leftarrow}(z)=\tau^{-\nu}z^{\nu}K^{\leftarrow}_{Q}(z).
\]
Solving the shift law \eqref{eq:E1} for the power factor, with $c=-\nu$, yields
\[
z^{\nu}K^{\leftarrow}_{Q}(z)
=e^{\pi\alpha\nu^{2}/\tau}
 K^{\leftarrow}_{Q-\nu}\!\left(e^{-2\pi\alpha\nu/\tau}z\right).
\]
Finally apply \eqref{eq:derivk} with $k=r$.  The chain-rule factor from the
scaled argument cancels the corresponding power in
$(e^{-2\pi\alpha\nu/\tau}z)^{-r}$, leaving the prefactor
$\tau^{-\nu}(-\tau)^r e^{\pi\alpha\nu^2/\tau}$ and the argument displayed in
\eqref{eq:RL}.  The parameter lists follow by first subtracting $\nu$ from the
Erd\'elyi--Kober parameters and then appending the two numerator $a$-parameters
and two denominator $b$-parameters of \eqref{eq:derivk}.  Invariance of $N$ is
Theorem~\ref{thm:Ninvariant}(i).
\end{proof}

\begin{remark}
The restriction $\sigma_{0}<1$ is genuine: it is
$\Re(\gamma+1)>\sigma_{0}$ with $\gamma=0$ in the integral leg. Integer
$\delta$ is handled directly by \eqref{eq:derivk}.  We do not state a Caputo
analogue here.  A direct contour proof for
${}^{C}D^{\delta}_{0+}=I^{\nu}_{0+}\circ d^{r}/dz^{r}$ requires at least
$\sigma_{0}<1-r$, as well as the separation conditions for the differentiated
kernel and the relevant endpoint data; without those additional hypotheses the
Riemann--Liouville and Caputo operators need not agree.
\end{remark}

\subsection{\texorpdfstring{$N$ is a strict invariant of the calculus}{N is a strict invariant of the calculus}}

\begin{theorem}\label{thm:Ninvariant}
\begin{enumerate}[label=\textup{(\roman*)},leftmargin=2.4em,itemsep=3pt]
\item Each of the four rules of Lemma~\ref{lem:insert} leaves both $N$ and
$p-q$ unchanged, and shifts $\mu$ and $\nu$ by $\pm\tau$ according to
\[
\Delta\mu=(+\tau,-\tau,-\tau,+\tau),\qquad
\Delta\nu=(-\tau,-\tau,+\tau,+\tau)
\]
for the rows in the order listed. Consequently the Euler, Erd\'elyi--Kober,
differentiation and Laplace insertions --- that is, every \emph{single-kernel}
operation of Theorem~\ref{thm:closure} --- preserve $N$ exactly.
\item Mellin convolution is the exception, because it multiplies two kernels
rather than inserting a multiplier into one: there $N=N_{1}+N_{2}$,
$\alpha=\alpha_{1}+\alpha_{2}$, and likewise for $\mu$, $\nu$, $\xi$.
\end{enumerate}
\end{theorem}

\begin{proof}
(i) Rows one and three insert one $b$ in the numerator block ($m\!+\!1$,
$q\!+\!1$) and one $a$ in the denominator block ($p\!+\!1$), so
$N\mapsto2(m+1+n)-(p+1)-(q+1)=N$ and $p-q\mapsto p-q$. Rows two and four insert
one $a$ in the numerator block ($n\!+\!1$, $p\!+\!1$) and one $b$ in the
denominator block ($q\!+\!1$), with the same conclusion. The displayed
$\Delta\mu,\Delta\nu$ follow from \eqref{eq:munuxi} by substituting the four
pairs of new parameters. (ii) Immediate from the concatenation of the two
parameter blocks in the kernel-product statement of
Theorem~\ref{thm:kernelcalc}.
\end{proof}

\begin{corollary}[Why every known example has $N=0$]\label{cor:Ntrap}
The embedding \eqref{eq:GammaviaG} sends every Meijer $G$-function --- hence
every ${}_pF_q$ --- to a $K$-function with $N=0$. By
Theorem~\ref{thm:Ninvariant}, starting from any finite collection of such
functions and applying any composition of the operations of
Theorem~\ref{thm:closure} together with \textup{(E1)}--\textup{(E3)}, one never
leaves the stratum $N=0$: the single-kernel operations preserve $N$ by
Theorem~\ref{thm:Ninvariant}(i), and Mellin convolution adds it by
Theorem~\ref{thm:Ninvariant}(ii), so $0+0=0$ closes the stratum under that
operation too. In particular:
\begin{enumerate}[label=\textup{(\alph*)},leftmargin=2.4em,itemsep=2pt]
\item all five applications in \cite[\S6]{KK2024} have $N=0$ --- as they do, by
direct inspection: $K^{1,1}_{2,2}$, $K^{1,1}_{2,2}$, $K^{1,2}_{3,3}$,
$K^{p+1,1}_{p+2,p+2}$, $K^{2,0}_{2,2}$;
\item a $K$-function with $N\ne0$, and in particular the $N<0$ functions of this
paper, cannot be produced from hypergeometric input by the classical transform
calculus. It must be assembled from double-gamma data directly --- for instance
from a Mellin transform satisfying $M(s+1)=R(s)M(s)$ with $R$ a ratio of gammas
whose numerator and denominator have \emph{different} lengths.
\end{enumerate}
\end{corollary}

\begin{remark}[The Fox $H$ boundary]\label{rem:foxH}
The restriction $\beta=1$ in Theorem~\ref{thm:closure}\,(ii) is essential. For
general $\beta$ the Erd\'elyi--Kober multiplier is
$\Gamma(\gamma+1-s/\beta)/\Gamma(\gamma+\delta+1-s/\beta)$, whose gamma
arguments are affine in $s$ with slope $1/\beta$, whereas the kernel
\eqref{eq:Gscript} admits only the slopes $\pm1$. This is precisely the
Meijer\,$G$ versus Fox\,$H$ dichotomy one rung down. The modular law
\textup{(E3)} does not repair it, because it rescales \emph{all} slopes
simultaneously and so cannot match a single mismatched factor. The object that
does accommodate independent slopes in a double-gamma kernel is the Fox--Barnes
$J$-function of Vaz \cite{Vaz2025}, whose slopes are arbitrary but positive ---
which is exactly the range $1/\beta>0$ needed here. The natural setting for
Erd\'elyi--Kober operators with $\beta\ne1$ is therefore that family rather than
this one.
\end{remark}

\begin{remark}[Degenerate pairs and cancelled poles]\label{rem:degenerate}
Every application of Lemma~\ref{lem:insert} creates a parameter pair
$\{c,\,c+\tau\}$ split between a numerator and a denominator block. The
numerator zeros then cancel a sublattice of the new poles: for the first row,
the new $a$-parameter $c$ contributes poles $\{c+l\tau+k\}_{l,k\ge0}$ while the
new $b$-parameter $c+\tau$ contributes zeros
$\{c+\tau+l\tau+k\}_{l,k\ge0}$, and the surviving poles are exactly
$\{c+k\}_{k\ge0}$ --- which are, as they must be, the poles of the multiplier
$\Gamma(c-s)$. This is the exact analogue of the Meijer\,$G$ reduction that
occurs when $a_j-b_k\in\Z_{\ge0}$; it means the genericity hypothesis
\eqref{eq:generic} \emph{fails by construction} after any transform. Practically,
the residue series of Theorem~\ref{thm:B} must skip the cancelled lattice
points, where the residue is $0$ rather than infinite; an implementation that
does not test for this raises a gamma-pole error instead of returning zero.
\end{remark}

\section{Numerical verification}\label{sec:numerics}

\subsection{\texorpdfstring{An evaluator for $G(z;\tau)$}{An evaluator for G(z;tau)}}
To the best of our knowledge no widely used numerical library provides
$G(z;\tau)$ for general $\tau$ (\texttt{mpmath.barnesg} is
the case $\tau=1$). We evaluate $\ln G$ by Lemma~\ref{lem:binet}, computing
$\int_{0}^{\infty}e^{-zx}f_{3}(x)dx$ in two independent ways --- by quadrature,
and by Watson's lemma
$\sum_{j\ge0}j!\,F_{j+3}\,z^{-(j+1)}$ where $F_{n}$ are the Taylor coefficients
of $f$ --- after first translating $z$ rightwards by the first relation of
\eqref{eq:quasiper}. Table~\ref{tab:V} lists the checks; all are performed at
$30$ decimal digits.

\begin{table}[t]
\centering\small
\caption{Validation of the $G(z;\tau)$ evaluator, at $30$ digits.
Worst relative error over the stated sample.}
\label{tab:V}
\begin{tabular}{@{}llc@{}}
\toprule
& check & worst rel.\ error\\
\midrule
V1 & $\tau=1$ against \texttt{mpmath.barnesg}, $7$ real and complex points
   & $2.8\times10^{-22}$\\
V2 & both quasi-periods \eqref{eq:quasiper}, $5$ values of $\tau$, $5$ complex $z$
   & $1.0\times10^{-21}$\\
V3 & modular transformation \eqref{eq:modular}, $4$ values of $\tau$
   & $9.8\times10^{-21}$\\
V4 & the gamma representation \eqref{eq:GammaviaG}, $5$ values of $\tau$
   & $1.0\times10^{-21}$\\
V5 & $b_{0}(\tau)$ by the asymptotic route vs.\ quadrature; $b_{0}(1)$ closed form
   & $6.4\times10^{-21}$\\
V6 & $G'$ at $9$ lattice zeros: \eqref{eq:Gprime} vs.\ a difference quotient
   & $4.7\times10^{-20}$\\
\bottomrule
\end{tabular}
\end{table}

\subsection{\texorpdfstring{The four $K$-functions}{The four K-functions}}
As a running example we take
\begin{equation}\label{eq:example}
m=n=0,\quad p=q=1,\quad a_{1}=0.3,\quad b_{1}=0.8,\quad
\tau=\tfrac{\sqrt5-1}{2},
\end{equation}
so that $N=-2$, $\lambda_{\sharp}=b_{1}-\tau-1=-0.8180\ldots<\lambda_{\flat}
=a_{1}=0.3$, and both $\alpha=0$ and $\alpha=0.25$ are tested. At $z=2$,
$\alpha=0$:
\[
\begin{aligned}
\Kr&=0.2505124042787052, &\qquad \Kl&=0.2624790203086704,\\
\Ku&=0.2564957122936878+0.0007915098718586\,i, &\qquad
\Kd&=\overline{\Ku},
\end{aligned}
\]
whose smallest pairwise relative gap is $6.2\times10^{-3}$: the four functions
are genuinely distinct, and $\Kr+\Kl=2\Rea\Ku$ to a relative error of
$8\times10^{-24}$. Figure~\ref{fig:kfour} shows the same picture over
$z\in[0.25,7.7]$.

\begin{figure}[t]
\centering
\includegraphics[width=\textwidth]{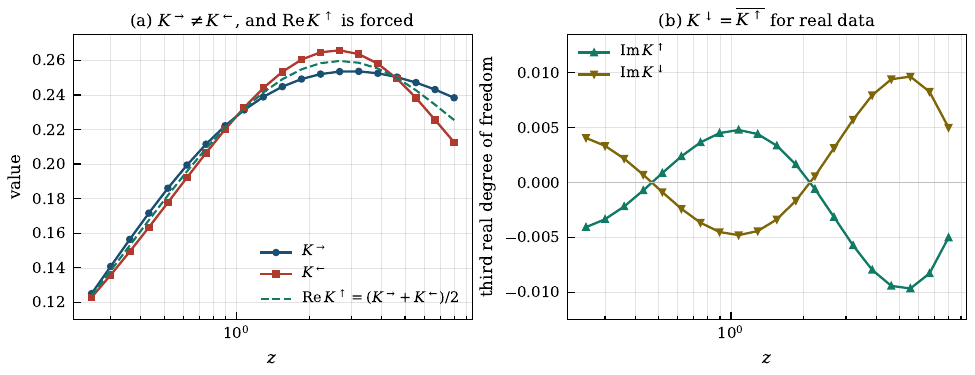}
\caption{The four $K$-functions for the parameters \eqref{eq:example},
$\alpha=0$. (a) $\Kr$ and $\Kl$ are distinct real functions; the dashed curve is
$\tfrac12(\Kr+\Kl)$, which Corollary~\ref{cor:real} forces to equal
$\Rea\Ku=\Rea\Kd$. (b) The remaining real degree of freedom,
$\Ima\Ku=-\Ima\Kd$. Over the whole sweep the relation \eqref{eq:relation} holds
to $3.7\times10^{-18}$.}
\label{fig:kfour}
\end{figure}

Table~\ref{tab:T} lists the checks of \S\S4--8; all are performed at $26$
decimal digits. Throughout, what is validated is a \emph{representative} set of
identities and parameter instances --- one $\tau$, one non-degenerate parameter
set, $\alpha\in\{0,0.25\}$ --- not the theorems in the generality in which they
are stated.

\begin{table}[t]
\centering\small
\caption{Verification of the results of this paper, at $26$ digits.
Worst relative error over the stated sample.}
\label{tab:T}
\begin{tabular}{@{}llc@{}}
\toprule
& check & worst rel.\ error\\
\midrule
T1 & homotopy invariance within each of the four classes ($h$, $\sigma_0$ varied)
   & $1.7\times10^{-21}$\\
T1b & the four values are pairwise distinct
   & gap $\ge6.2\times10^{-3}$\\
T1c & Corollary~\ref{cor:real}: $\Kr,\Kl\in\R$ and $\Kd=\overline{\Ku}$
   & exact\\
T2 & Theorem~\ref{thm:C}, each term computed on a \emph{different} realisation
   & $1.0\times10^{-21}$\\
T3 & Theorem~\ref{thm:B}: residue series vs.\ contour integral
   & $1.8\times10^{-17}$\\
T4 & Theorem~\ref{thm:D}: vanishing, relative to the surviving partner
   & $7.7\times10^{-22}$\\
T5a & Theorem~\ref{thm:E}(E1), all four classes, $\alpha\in\{0,0.25\}$
   & $2.0\times10^{-27}$\\
T5b & Theorem~\ref{thm:E}(E2), all four pairings
   & $1.6\times10^{-27}$\\
T5c & Theorem~\ref{thm:E}(E3), all four classes, $\alpha\in\{0,0.25\}$
   & $5.4\times10^{-21}$\\
T6 & sign audit of \cite[Remark after Thm.~5]{KK2024}, see \S\ref{sec:errata}
   & ratio $=-1$\\
T7 & Theorem~\ref{thm:F}, $3\times3$ grid on $\Lp$ and $2\times2$ on $\Lm$
   & $2.8\times10^{-23}$\\
\bottomrule
\end{tabular}
\end{table}

\subsection{The transform calculus}
The closure results of \S\ref{sec:closure} are checked in two independent ways.
First, at the level of the identity \eqref{eq:insert} itself, which is exactly
the parameter bookkeeping and can be tested to full precision at complex $s$
without ever evaluating a $K$-function. Second, end-to-end in $x$-space, where
the left-hand side is an honest quadrature of $K^{\bullet}$ and the right-hand
side an independently computed $K$-function --- itself evaluated both by its
residue series and by its contour integral. Table~\ref{tab:C} reports both, for
the running example \eqref{eq:example} ($N=-2$), $\alpha\in\{0,0.25\}$, and
$\beta=\tfrac32$, $\gamma=2$, $\delta=\tfrac7{10}$. The default abscissa there is
$\sigma_{0}=\tfrac12(\lambda_{\sharp}+\lambda_{\flat})=-0.2590\ldots$, so
$\Rea\beta>\sigma_{0}$ and $\Rea(\gamma+1)>\sigma_{0}$ and both
\eqref{eq:eulerhyp} and \eqref{eq:ekhyp} hold on the sample.

\begin{table}[t]
\centering\small
\caption{Verification of \S\ref{sec:closure}, at $25$ digits. Worst relative
error over the stated sample.}
\label{tab:C}
\begin{tabular}{@{}l@{\ \ }p{0.60\textwidth}@{\ \ }c@{}}
\toprule
& check & worst rel.\ error\\
\midrule
CA & Lemma~\ref{lem:insert}, all four rules; $3$ values of $c$, $5$ complex
     $s$, $\alpha\in\{0,0.25\}$ & $7.0\times10^{-23}$\\
CA2 & the composites: Euler, Erd\'elyi--Kober, and the multiplier $-s$ of
      \eqref{eq:deriv} & $1.1\times10^{-22}$\\
CD & Theorem~\ref{thm:Ninvariant}: $\Delta N=\Delta(p-q)=0$,
     $|\Delta\mu|=|\Delta\nu|=\tau$; $N$ additive under \eqref{eq:mconv}
   & exact\\
CB & Theorem~\ref{thm:closure}(i) end-to-end: quadrature of \eqref{eq:euler}
     vs.\ the new $K^{\leftarrow}$, $z\in\{0.6,1.3\}$ & $2.1\times10^{-23}$\\
CC & Theorem~\ref{thm:closure}(iii) end-to-end: $K'$ by central difference vs.\
     \eqref{eq:deriv}, $\bullet\in\{\rightarrow,\leftarrow\}$
   & $2.1\times10^{-13}$\\
CE1 & \eqref{eq:derivk} for $k=1,2,3$ vs.\ \texttt{mp.diff}, and the two
      realisations of Remark~\ref{rem:tworealisations} against each other
   & $9.2\times10^{-23}$\\
CE2 & Theorem~\ref{thm:RL}: $D^{\delta}_{0+}K^{\leftarrow}$ leg by leg,
      $\delta\in\{0.7,1.4\}$ and $\alpha\in\{0,0.25\}$ & $1.5\times10^{-8}$\\
\bottomrule
\end{tabular}
\end{table}

Two remarks on CB and CC. The constants come out as predicted and are worth
recording because they are sharp tests of the bookkeeping: for
$\tau=\tfrac{\sqrt5-1}{2}$ the Euler prefactor is
$\Gamma(\gamma)\tau^{-\gamma}=\tau^{-2}=2.6180339887\ldots$ and the
differentiation prefactor is exactly $\tau$, while in both cases the argument
scale is $\lambda=1$, the two factors $\tau^{\pm s}$ cancelling. The residual
$2.1\times10^{-13}$ in CC is the central-difference step, not the identity.

Finally, CB is the check that exposed Remark~\ref{rem:degenerate}: with the
transformed parameters the pair $\{\beta,\beta+\tau\}$ makes a numerator factor
vanish on part of $\Lp$, and a residue routine that does not detect this fails
with a gamma-pole error rather than returning zero.

\subsection{The four points the proofs turn on}
Four of the statements above have delicate proofs, and each has a dedicated
check. Table~\ref{tab:R} reports them.

\begin{table}[t]
\centering\small
\caption{Checks bearing directly on the hypotheses of
\S\ref{sec:arcs}--\S\ref{sec:closure}, at $22$ digits.}
\label{tab:R}
\begin{tabular}{@{}llc@{}}
\toprule
& check & result\\
\midrule
R1 & recurrence compatibility \eqref{eq:compat}
   & $3.0\times10^{-22}$\\
R2a & the four classes on \emph{slanted} rays: vs.\ horizontal, and vs.\ a
      second $\eps$ & $5.3\times10^{-23}$\\
R2b & relation \eqref{eq:relation} computed entirely on slanted rays
   & $6.2\times10^{-23}$\\
R3 & evaluation determinant, $(\Kr,\Kl,\Ku)$ at $z\in\{0.5,1.6,4.0\}$
   & $|\det|/\max^{3}=2.4\times10^{-3}$\\
R4 & closing-arc bound \eqref{eq:arcbound}, sup form
   & see below\\
R5 & adaptive radius: mean of the arc integral over $r\in[R,2R]$,
     \eqref{eq:Mdef} & see below\\
R6a & the split \eqref{eq:split}: $\ln|\Pi|$ from $S_{2}$ vs.\ by subtraction
   & $7.6\times10^{-21}$\\
R6 & growth of $\max_{\theta}\ln|\Pi|$ on the eastern arcs, $R\le12$
   & bounded; see below\\
\bottomrule
\end{tabular}
\end{table}

R2 is the check that matters for Theorem~\ref{thm:C}: the relation is verified on
a realisation in which \emph{every} piece lies inside a decay sector, so it does
not inherit the difficulty of \S\ref{sec:arcs}. R4 probes \eqref{eq:arcbound}
at four radii: the quantity
$\min_{R'\in[R,R+1]}\max_{|\arg s|\le\eps_{1}}|\phi(s)z^{-s}|$ takes the values
$2.1\times10^{-6}$, $1.1\times10^{-22}$, $1.2\times10^{-52}$,
$6.4\times10^{-97}$ at $R=3,5,7,9$. This is strong numerical support for the
$\exp(-cR^{2}\ln R)$ behaviour Lemma~\ref{lem:arcs} asserts, for this parameter
set and this $\tau$; it is not a verification of an asymptotic hypothesis.

R5 tests the hypothesis Lemma~\ref{lem:adaptive} actually uses. The mean of the
eastern-arc integral over $r\in[R,2R]$, sampled at thirteen radii, is
$1.4\times10^{-4}$, $5.7\times10^{-16}$ and $9.7\times10^{-40}$ at $R=3,5,7$;
the spread between the smallest and largest sampled value is enormous
($3.9\times10^{21}$ already at $R=3$), but that spread is dominated by the
$r$-dependence of $\exp(-cr^{2}\ln r)$ across the annulus rather than by proximity
to poles. R6a confirms the algebraic split \eqref{eq:split} to
$7.6\times10^{-21}$, which is what licenses reading $\Pi$ off by subtraction; R6
then measures $\max_{\theta}\ln|\Pi|$ on the eastern arcs and finds it
\emph{bounded} --- $1.41$, $0.52$, $-0.13$, $0.12$, $0.45$ at $R=4,6,8,10,12$ ---
so on this sample \Harc holds with exponent $0$. R3 is evidence for $\dim=3$ and
is labelled as such (Remark~\ref{rem:three}).

\subsection{Three silent failure modes}\label{sec:pitfalls}
Each of the following returns a perfectly smooth but wrong answer rather than an
error, and each cost us several digits before being diagnosed.

\begin{enumerate}[label=\textup{(P\arabic*)},leftmargin=2.6em,itemsep=4pt]
\item \emph{The separating abscissa.} The vertical connector must satisfy
$\lambda_{\sharp}<\sigma_{0}<\lambda_{\flat}$. Placing it outside does not make
the integral diverge; it silently computes a \emph{different} $K$-function. The
condition should be asserted in code, not assumed.

\item \emph{Panel width versus pole distance.} For real parameters all poles lie
on $\R$, at distance $|h|$ from the horizontal ray. A Gauss--Legendre panel
wider than $|h|$ places a singularity inside its Bernstein ellipse. Measured:
uniform panels of width $1.5$ at $h=0.25$ lose four digits; panels refined only
near the turning point lose three. Panel width $\approx|h|$ restores full
accuracy.

\item \emph{Fixed ray truncation fails once $\alpha\ne0$.} The factor
$e^{\pi\alpha s^{2}/\tau}$ grows like $e^{\pi\alpha u^{2}/\tau}$ along the ray
and is beaten by the decay $e^{-\frac{|N|}{2\tau}u^{2}\ln u}$ only once
$\ln u>\pi\alpha/|N|$; before that the integrand has a large hump.
Figure~\ref{fig:hump} shows a case whose value is $0.2046$ and whose integrand
peaks at $149$ near $u=-5.3$, still $2.3$ at $u=-8$: truncating there gives an
answer that is $12\%$ wrong and looks perfectly converged. The remedy is to
march outwards until the magnitude has passed its maximum \emph{and} fallen
below a prescribed fraction of it.
\end{enumerate}

\begin{figure}[t]
\centering
\includegraphics[width=0.56\textwidth]{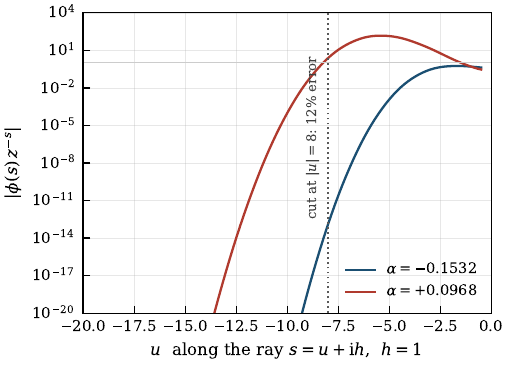}
\caption{Magnitude of the integrand along the western ray for the modular image
(E3) of the example \eqref{eq:example} at $z=1.6$, $h=1$. For $\alpha<0$ the
decay is monotone; for $\alpha>0$ the Gaussian factor produces a hump of height
$149$ that has not decayed at $|u|=8$. A fixed truncation there is $12\%$ wrong.}
\label{fig:hump}
\end{figure}

A fourth trap is not specific to this problem but was expensive: a memoisation
cache over a multiprecision special function must include the working precision
in its key. Caching $\ln G$ without it returned $20$-digit values to a
$30$-digit caller and cost seven digits, silently.

\section{\texorpdfstring{A corrigendum to \cite{KK2024}}{A corrigendum}}\label{sec:errata}

The single point below was found while implementing the results above and is
confirmed numerically. It is present in the published version --- everything
here was checked against \emph{Constr.\ Approx.} \textbf{63} (2026), 223--256
rather than against the preprint --- and it is a sign, not a typographical
ambiguity.

\subsection*{The sign in the Remark following Theorem 5}
That Remark \cite[p.~244]{KK2024} states, for a simple pole $q_{k}\in\Lp$ of minimal real part,
\[
B_{k,0}=\frac{\tau\prod_{j=1}^{m}G(b_j-q_k;\tau)\prod_{j=1}^{n}G(1+\tau-a_j+q_k;\tau)}
     {\prod_{j=m+1}^{q}G(1+\tau-b_j+q_k;\tau)\sideset{}{'}\prod_{j=n+1}^{p}G(a_j-q_k;\tau)}
     e^{\pi\alpha q_{k}^{2}/\tau},
\]
in conjunction with $K=-\big[\sum_k z^{-q_k}B_{k,0}\big](1+\cdots)$. The
vanishing factor at $s=q_k=a_j$ is $G(a_j-s;\tau)$, whose argument has
derivative $-1$ with respect to $s$; with Lemma~\ref{lem:Gzero} this gives
\[
\Res_{s=a_j}\frac{1}{G(a_j-s;\tau)}=-\tau,\qquad\text{not}\qquad+\tau .
\]
The constant should therefore read $-\tau$ (equivalently, the minus sign in the
displayed formula of \cite[Thm.~5]{KK2024} should be deleted). The Remark
following \cite[Thm.~4]{KK2024} is \emph{correct} as printed: there the
vanishing factor is $G(1+\tau-b_j+s;\tau)$, whose argument has derivative $+1$.

\emph{Numerical confirmation.} Take $(m,n,p,q)=(1,0,1,1)$, $a_{1}=0.3$,
$b_{1}=1.1$, $\tau=(\sqrt5-1)/2$, $\alpha=0$, $z=6$. Then $N=0$, $p=q$,
$\mu=b_{1}-a_{1}=0.8>0$, so this is the balanced case (v) of \cite{KK2024},
whose Table~1 prescribes $\theta^{+}=\eps$, $\theta^{-}=-\eps$ --- precisely the
right hairpin. Computing the contour integral and the residue sum
independently,
\[
K=0.279968197413998,\qquad
-\!\!\sum_{\Lp}\Res=0.279968197413998
\quad(\text{rel.\ }2.2\times10^{-22}),
\]
while the leading terms are
\[
-\Res_{s=a_1}=+0.35441655327321,
\qquad
\text{Remark as printed}=-0.35441655327321,
\]
a ratio of exactly $-1$.

\subsection*{Verified correct}
For the record, the following were checked pointwise to better than
$10^{-19}$ and are correct as printed: the quasi-periodicities
\cite[Eq.~(4)]{KK2024}; the modular transformation \cite[Eq.~(5)]{KK2024}, whose
exponents are unambiguously those of \eqref{eq:modular}; the transformation
\cite[Eq.~(16)]{KK2024} of $\Gs$ and its constants \cite[Eq.~(17)]{KK2024}; and
the gamma representation \cite[Eq.~(27)]{KK2024}; and the prefactor $A\tau$ of
\cite[Eq.~(30)]{KK2024}, which is the product $A\cdot\tau$, is typeset
unambiguously as such, and was verified on all four contour classes for
$\alpha\in\{0,0.25\}$ (T5c, $5.4\times10^{-21}$).

\section{Concluding remarks and open problems}\label{sec:open}

\begin{enumerate}[label=\textup{(O\arabic*)},leftmargin=2.6em,itemsep=5pt]

\item\label{op:diophantine} \emph{The closing-arc input.} Hypothesis \Harc of
\S\ref{sec:arcs}, on which Theorems~\ref{thm:B} and \ref{thm:D} rest, is a
growth bound on a product of double sines evaluated on $\Z+\tau\Z$, and it
depends on $\va,\vb$ as well as on $\tau$. We prove it for no $\tau$ at all.
Two reductions have been carried out and a third step is missing.
\emph{Done:} the exponent needed is only $o(R^{2}\ln R)$, not the
$O(R\ln R)$ one might guess from the small-denominator literature; and the
prescribed sequence of radii can be dispensed with, since
Lemma~\ref{lem:adaptive} produces it from an $L^{1}$ bound on annuli.
\emph{Missing:} an upper bound of size $\exp(o(R^{2}\ln R))$ for
$\iint_{A(R)}|\Pi|\,dA$. The obstruction is circular in a specific way: near a
pole $s_{0}$ of $\Pi$ one has $|\Pi(s)|\le C_{s_{0}}|s-s_{0}|^{-1}$ with
$C_{s_{0}}$ itself a product of reciprocal sines over the remaining lattice
points, so controlling the area integral requires controlling exactly the
quantity one set out to control. Breaking that circle --- or replacing the
crude pointwise bound by a mean-value argument that exploits
$\int_{0}^{1}\ln|2\sin\pi t|\,dt=0$ --- would make Theorems~\ref{thm:B} and
\ref{thm:D} unconditional and is the single most valuable addition this paper
still needs. A separate and easier question: find number-theoretic conditions on
$(\tau,\va,\vb)$ sufficient for \Harc. The analogous quantity controls the
convergence of the residue double series (Remark~\ref{rem:diophantine}) and, one
rung down, the series representation of the density of the supremum of a stable
process \cite{HubalekKuznetsov2011,Kuznetsov2013,HackmannKuznetsov2013}, but
those results are for a specific coefficient structure and do not transfer as they
stand.

\item\label{op:dim} \emph{Is the dimension exactly three?} We prove only
$\dim\operatorname{span}\{\Kr,\Kl,\Ku,\Kd\}\le3$ (Theorem~\ref{thm:C}); the
numerical determinant of Remark~\ref{rem:three} is evidence, not proof. A proof
would exhibit three asymptotic regimes in which $\Kr$, $\Kl$ and $\Ku$ have
distinct leading behaviour, with an invertible coefficient matrix.

\item\label{op:Ahat} \emph{Identify $\widehat A$.} The transcendent
$\widehat A=\Ku-\Kr$ of Theorem~\ref{thm:C} carries all the information not
contained in the two residue sums. Is there a closed form, for instance in terms
of the double sine function $S_{2}$?

\item\label{op:mellin} \emph{Does a Mellin transform exist at all?} We prove only
that no \emph{vertical} contour is admissible (Theorem~\ref{thm:F1}). Whether
$\int_{0}^{\infty}K^{\bullet}(x)x^{s-1}dx$ converges for some $s$ in the
non-degenerate case requires genuine two-sided asymptotics for $K^{\bullet}$ and
the non-vanishing of the leading coefficients; see Remark~\ref{rem:notclaimed}.

\item\label{op:transforms} \emph{Laplace transform and Mellin convolution.} For
these two the kernel identity of Theorem~\ref{thm:kernelcalc} holds but the
$x$-space identity is not established (Theorem~\ref{thm:closure}(iv)). What is
missing is a substitute for the fundamental strip: an analytic continuation
argument, or a distributional formulation, in which the two convolutions make
sense for $N<0$.

\item\label{op:complex} \emph{Complex parameters.} Our implementation assumes
$a_j,b_j\in\R$, so that all poles are collinear. For complex parameters
Proposition~\ref{prop:combs} still confines them to finitely many horizontal
lines and the theory is unchanged, but the contour must weave between those
lines and the closing arcs must be chosen accordingly.

\item\label{op:Nneg} \emph{Occurrences of $N<0$.} All five applications in
\cite[\S6]{KK2024} --- extrema of stable processes, eigenfunctions of the
fractional Laplacian on a half-line, exponential functionals of hypergeometric
L\'evy processes, the generalised Kilbas--Saigo function, and Barnes beta
distributions --- have $N=0$, and Corollary~\ref{cor:Ntrap} explains why: the
classical transform calculus preserves $N$, and hypergeometric input enters at
$N=0$. An $N\ne0$ example must therefore be assembled from double-gamma data
directly. Concretely, one needs a Mellin transform obeying $M(s+1)=R(s)M(s)$
where $R$ is a ratio of gamma functions with \emph{unequal} numbers of factors
upstairs and downstairs; the recursions in
\cite{Kuznetsov2011,KuznetsovPardo2013,Ostrovsky2013} are all balanced. Finding
an unbalanced one --- or proving that unbalanced recursions cannot arise from a
L\'evy or Markov structure --- would be the most interesting sequel.

\item\label{op:rest} \emph{The other excluded cases.} Besides $N<0$,
\cite{KK2024} excludes $N=0$ with $\Rea\alpha<0$, and three boundary
configurations. The first should yield to the same analysis --- there
$\Rea\alpha<0$ turns the vertical direction into a growth direction while $N=0$
leaves the horizontal one marginal --- but the marginal decay makes both the
contour classification and the closing arcs more delicate, and we leave it open.
\end{enumerate}

\section*{Reproducibility}
All computations reported here were performed with \texttt{mpmath}
\cite{mpmath} at the stated working precisions. Five primary modules are used, and each
is self-validating --- running it reproduces the corresponding table:
an evaluator for $G(z;\tau)$ at general $\tau>0$ implementing
Lemmas~\ref{lem:Gprime} and \ref{lem:binet} (Table~\ref{tab:V}); the four
$K$-functions of Definition~\ref{def:four} with their residue series and the
recurrence \eqref{eq:recur} (Table~\ref{tab:T}); the insertion
rules of Lemma~\ref{lem:insert} and the transforms of
Theorem~\ref{thm:closure} (checks CA--CC of Table~\ref{tab:C}); the slanted-ray
realisation, the evaluation determinant and the closing-arc probes (checks
R1--R4 of Table~\ref{tab:R}); and the derivative and fractional-derivative
formulas \eqref{eq:derivk}, \eqref{eq:RL} together with the adaptive-radius and
double-sine measurements (checks CE1, CE2 and R5, R6 of the same two tables).
Two additional targeted regressions, \texttt{v7\_fast\_checks.py} and
\texttt{rl\_alpha\_nonzero\_check.py}, independently check the corrected
$\alpha$-dependent shift and the nonzero-$\alpha$ Riemann--Liouville formula.
The complete validation transcript is supplied with the source package.

\subsection*{Acknowledgements}
The author thanks D.~Karp and A.~Kuznetsov, whose paper \cite{KK2024} both posed
this question and supplied every tool needed to answer it.


\end{document}